\documentclass[reqno]{amsart}

\numberwithin{equation}{section}
\usepackage{latexsym}
\usepackage{amsmath}
\usepackage{amssymb}
\usepackage{mathrsfs}
\usepackage{graphicx,colordvi}
\usepackage{upgreek}
\usepackage{ifthen}
\usepackage{color}

\usepackage[T1]{fontenc}
\usepackage[latin1]{inputenc}

\numberwithin{equation}{section}

\newtheorem{defi}{Definition}[section]
\newtheorem{theorem}[defi]{Theorem}
\newtheorem{lemma}[defi]{Lemma}

\newtheorem{proposition}[defi]{Proposition}
\newtheorem{remark}[defi]{Remark}
\newtheorem{remarks}[defi]{Remarks}

\usepackage{latexsym}
\usepackage{amsmath}
\usepackage{amssymb}

\newcommand{\cE}{{\mathcal E}}
\newcommand{\cH}{{\mathcal H}}
\newcommand{\cB}{{\mathcal B}}

\newcommand{\C}{{\mathbb C}}
\newcommand{\CC}{{\mathbb C}}

\newcommand{\BB}{{\mathbb B}}
\newcommand{\EE}{{\mathbb E}}

\newcommand{\N}{{\mathbb N}}
\newcommand{\II}{{\mathbb I}}

\newcommand{\R}{{\mathbb R}}

\newcommand{\PP}{{\mathbb P}}

\providecommand{\norm}[1]{\left\vert #1 \right\vert}

\renewcommand{\epsilon}{\varepsilon}

\newcommand{\sa}{{\sf a}}
\newcommand{\bb}{{\bb}}
\newcommand{\eps}{\varepsilon}

\begin{document}

\title[Motion of a rigid body with a fluid-filled cavity]{A maximal regularity approach to the study of motion of a rigid body with a fluid-filled cavity}

\author{Giusy Mazzone}
\address{Department of Mathematics\\
        Vanderbilt University\\
        Nashville, Tennessee\\
        USA}
\email{giusy.mazzone@vanderbilt.edu}

\author{Jan Pr\"uss}
\address{Martin-Luther-Universit\"at Halle-Witten\-berg\\
         Institut f\"ur Mathematik \\
         Theodor-Lieser-Strasse 5\\
         D-06120 Halle, Germany}
\email{jan.pruess@mathematik.uni-halle.de}

\author{Gieri Simonett}
\address{Department of Mathematics\\
        Vanderbilt University\\
        Nashville, Tennessee\\
        USA}
\email{gieri.simonett@vanderbilt.edu}

\thanks{This work was supported by a grant from the Simons Foundation (\#426729, Gieri Simonett).}

% 35K58 Semilinear parabolic equations 
% 35B35 Stability 
% 35B40 Asymptotic behavior of solutions 
% 35Q30 Navier-Stokes equations
% 35Q35 PDEs in connection with fluid mechanics 


\begin{abstract}
We consider the inertial motion of a rigid body with an interior cavity that is completely filled with a viscous incompressible fluid.
The equilibria of the system are characterized and their stability properties are analyzed.
It is shown that equilibria associated with the largest moment of inertia are normally stable,
while all other equilibria are normally hyperbolic. 

We show that every Leray-Hopf weak solution
converges to an equilibrium at an exponential rate. 
In addition, we determine the critical spaces for the governing evolution equation, and we demonstrate how parabolic regularization in time-weighted spaces affords great flexibility in establishing regularity of solutions and their convergence to equilibria. 
\end{abstract}

\maketitle

\section{Introduction and formulation of the problem}
Consider the system $\mathcal S$ constituted by a rigid body with a hollow cavity $\Omega$ completely filled by a viscous incompressible fluid. We investigate the stability properties and long-time behavior of the whole system fluid-filled rigid body with respect to a coordinate system that is attached to
the center of mass $G$ of $\mathcal S$. In absence of external forces, the motion of the coupled system is driven by an initial angular momentum imparted on $\mathcal S$ ({\em inertial motion}).  

The problem of stability for physical systems like $\mathcal S$ has caught the attention of researchers in 
different fields of the applied sciences. Studies related to this problem can be found in the engineering literature 
concerning the dynamics of flight  (see e.g. \cite{Sc,Ka}), in structural mechanics 
(see e.g. \cite{SaTaTa,SaTaTa2,Taetal}), and space technology (\cite{Ibr,Ab}). 
Besides the interest for the applications, there have been numerous mathematical contributions which can be 
tracked back to the work by Stokes \cite{Sto}, 
Zhukovskii \cite{Zh}, Hough \cite{Ho}, Poincar\'e \cite{Po}, and Sobolev \cite{So}, mostly concerned with 
{\em ideal} fluids. In recent years, new stability results have been derived: they are obtained for suitable geometrical configurations of $\mathcal S$ and/or by linearizing the equations of motion (\cite{Ru2,Ru3,Ru4,Ch,Smi,Lya,KoShYu,MoRu,KoKr00}). In \cite{SiTa,Ma12},  the class of 
weak solutions {\em \`a la Leray-Hopf}, corresponding to initial data having finite total kinetic energy, has been proved to be nonempty. Moreover, in the Leray-Hopf class, weak solutions are in fact strong if small initial data are considered
(the smallness is in a suitable sense for both the initial relative velocity 
and the initial total angular momentum). In this class of strong solutions, the relative velocity of the 
fluid decays to zero, as $t\to \infty$, in the $L_2$-norm. Furthermore, in the case of {\em spherical 
mass symmetry}\footnote{In this case the inertia tensor of $\mathcal S$, calculated with respect to the center of mass G, is a multiple of the identity tensor. } of $\mathcal S,$ it has been proved that the decay of the fluid energy is, indeed, exponential (\cite{Ma12}).   Furthermore, the long-time behavior of $\mathcal S$ is characterized by a rigid body motion with constant angular velocity.

A more complete description of the long-time behavior of $\mathcal S$ has been given in \cite{GaMaZuCRM,DGMZ16,Ma16,Ga17}. It has been proved that 
each Leray-Hopf solution, as $t\to \infty$, must 
converge (in a proper topology) to a manifold $\{v \equiv 0\} \times \mathcal A$, where $\mathcal A$ is 
a compact, connected subset of $\R^3$ constituted by vectors (angular velocities) 
whose magnitude is compatible with conservation of total angular momentum. In particular, 
$\mathcal \{0\}\times \mathcal A\subset \mathcal E$, 
where $\mathcal E$ denotes the set of equilibria for the system fluid-filled rigid body $\mathcal S$. 
Let $\lambda_i$, $i = 1,2,3$,  be the central moments of 
inertia of $\mathcal S$ (i.e. the moments of inertia of $\mathcal S$ calculated with respect  to its center of mass $G$).
It has been demonstrated that if either $\lambda_1\le \lambda_2<\lambda_3$, 
or $\lambda_1=\lambda_2=\lambda_3$, 
then each weak solution converges to an equilibrium (\cite{GaMaZuCRM,DGMZ16,Ma16}). These results do not provide a rate for the convergence to an equilibrium, and the case $\lambda_1<\lambda_2=\lambda_3$ was left open. In these papers, also attainability and nonlinear stability of the equilibria has been studied. Recently, in \cite{Ga17}, a proof of convergence of Leray-Hopf solutions to a corresponding equilibrium also in the case $\lambda_1< \lambda_2=\lambda_3$ has been provided.
For the fluid velocity relative to the solid, exponential decay has been proved in the topology of $H^{2\alpha}_2$ 
(which coincides with the domain of fractional power of the Stokes operator). In the same functional setting,  asymptotic stability and instability properties are analyzed in \cite{Ga17}. Moreover, decay in stronger norms  is also obtained if additional conditions are imposed on the initial data. 

In the present paper, we provide a comprehensive study of the motion of fluid-filled rigid bodies in a different and general perspective. 
We show that equilibrium configurations corresponding to {\em permanent rotations}
\footnote{These are rigid body motions  around 
the principal axes of inertia, with constant angular velocity. 
For this motion, the fluid is at relative rest with respect to the solid. } of $\mathcal S$ around the central axes of inertia corresponding to the largest moment of inertia are asymptotically (actually, exponentially) stable. All other permanent rotations are unstable. These properties have been proved in \cite{Ga17} in an $L_2$-setting. Our stability results are  obtained in an $L_q$-framework, 
and as a byproduct of  the {\em generalized principle of linearized stability} introduced in \cite{PrSiZa09} (see also \cite[Chapter~5]{PrSi16}). The main ingredients that allow us to obtain such a result rely on the fact that the 
set of equilibria $\mathcal E$ forms a finite dimensional manifold, with dimension $m=1,2,3,$ depending on the mass distribution of $\mathcal S$ (see \eqref{equilibria2}). Moreover, a detailed study of the spectrum of the linearization  around a nontrivial equilibrium shows that equilibria associated with the largest moment of inertia are normally stable, while all other equilibria are normally hyperbolic (see Theorem \ref{th:spectrum}). 

The equations of motion given in \eqref{problem} form a coupled system of nonlinear parabolic PDEs and ODEs with bilinear nonlinearities. Using parabolic regularization in time-weighted spaces (\cite{PrSiWi18,PrSi16}), we show that, for any initial data having finite energy, all corresponding weak solutions will converge to an equilibrium at an exponential rate, in the topology of $H^{2\alpha}_q$ with $\alpha \in [0,1)$ and $q\in (1,3)$. The latter is shown by proving that, for each Leray-Hopf solution, there exists a time $\tau>0$ after which the solution becomes regular in suitable time-weighted spaces. 
We determine the critical spaces for the governing evolution equation for $\mathcal S$. The functional setting we use is that of maximal $L_p-L_q$ regularity in time-weighted $L_p$ spaces (\cite{PrSi16}). In Theorem \ref{th:local_strong}, we determine the largest space of initial data for which the equations of motion are well-posed. Finally, we use the important information (proved in \cite{DGMZ16}) that Leray-Hopf solutions become strong 
(in the $L_2$-setting for the fluid velocity and $C^1$ in time for the angular momentum), and moreover, the fluid relative velocity decays to zero, as $t\to\infty$, in the Sobolev $H^1_2$-norm. The above time $\tau$ can be thus found as the time when the Leray-Hopf solution has gained enough regularity, so that a point of its trajectory (corresponding to time $\tau$) can be considered as initial condition for a strong solution constructed in our Theorem~\ref{th:local_strong}. Well-posedness in $[\tau, \infty)$ then follows by choosing a suitable time-weight which ensures that weak and strong solutions coincide 
and trajectories are relatively compact. 

%%%%%%%%%%%%%%%
The approach taken in this paper makes systematic use of the theory of maximal regularity in time weighted spaces,
which has proven itself to be a powerful and flexible tool for the treatment of nonlinear parabolic equations, see for instance \cite{PrSi16, PrSiWi18}.  
The results contained in Proposition~\ref{prop:equilibria}(a)-(c) were already obtained in~\cite{GaMaZuCRM}, whereas the 
variational characterization in 
parts (d)-(e) of the proposition is new. Proposition~\ref{pro:L-calculus} and the well-posedness results in 
Theorem~\ref{th:local_strong} are also new, see Remarks~\ref{remark-2}. 
The characterization of the spectrum of $L_*$ given in Theorem~\ref{th:spectrum} parallels the results in \cite[Proposition~4.6]{Ga17}. However, we give a different proof which has more of a geometric flavor. In particular, the proof
of parts (d) and (e) adds geometric insight into the structure and the mechanism of occurrence of unstable eigenvalues. 
The assertions of Proposition~\ref{pro: stability-instability} are implicitly contained in the work of Prüss and Simonett, 
but appear here in print for the first time.
Finally, Theorem~5.2 and Theorem~\ref{th:convergence} are obtained as direct application of the generalized principle of linearized stability established in~\cite{PrSiZa09}.  Theorem~5.2 and Theorem~\ref{th:convergence} provide a generalization of results previously obtained in the $L_2$-framework (see \cite[Theorem~4.11 and Theorem~4.17, respectively]{Ga17}). 
%%%%%%%%%%%%%%%

Here is the plan of our paper. Below, we will provide the mathematical formulation of the problem. In Section \ref{se:energy and equilibria}, we will introduce the available energy of our system and characterize the set of equilibria.  
The well-posedness of the governing equations and the relevant critical spaces will be discussed in 
Section~\ref{se:local well-posedness and critical spaces}. In Section~\ref{se:spectrum}, we will analyze the spectrum of the linearization $L_*$ at an equilibrium, and provide a complete characterization of the nontrivial equilibria as either normally stable or normally hyperbolic. The stability properties of our equilibria, at the nonlinear level, is studied in Section \ref{sec:stability}.  We conclude the paper with Section \ref{se:long-time} in which we characterize the long-time behavior of $\mathcal S$. 

The equations governing the motion of the system fluid-filled rigid body $\mathcal S$ in a {\em non-inertial frame} with origin at $G$
(the center of mass of the whole system), and axes ${\sf e}_i$, $i=1,2,3$, directed along the principal axes of inertia of $\mathcal S$, read as follows (see \cite{DGMZ16,Ma12}): 
\begin{equation}
\label{problem}
\begin{aligned}
\partial_t v + v\cdot\nabla v +(\dot\sa -\dot\omega)\times x + 2(\sa-\omega)\times v
-\upnu \Delta v  +\nabla p & =0 &&\text{in} \;\; \Omega\times\R_+, \\
{\rm div}\, v &= 0 &&\text{in} \;\; \Omega\times\R_+, \\
v &=0 &&\text{on} \;\; \partial\Omega\times\R_+, \\
\II\, \dot\sa + (\sa-\omega)\times\II\sa &=0 &&\text{on}\;\; \R^3\times\R_+,\\
(v(0),\sa(0))= (v_0, &\sa_0)  &&\text{in} \;\; \Omega\times \R^3.\\
\end{aligned}
\end{equation}
Here  $\Omega\in \R^3$ is a bounded domain with boundary $\partial \Omega$ of class $C^{3}$, $v$ denotes the fluid velocity relative to the rigid body, $\upnu$ is its coefficient of kinematic viscosity, and $p$ the pressure field. 
Moreover, $\II={\rm diag}[\lambda_1,\lambda_2,\lambda_3]$ denotes the inertia tensor of $\mathcal S$ with respect to $G$, $\lambda_j$
 are the central moments of inertia of $\mathcal S$, $\II\sa $ is the total angular momentum of $\mathcal S$ with respect to $G$, and 
\begin{equation}
\label{def-omega}
\omega :=\II^{-1}\int_\Omega x \times v\,dx.
\end{equation}
Without loss of generality, we have set the fluid density $\varrho\equiv 1$. 
We note that the total angular momentum is a conserved quantity:
\begin{equation}
\label{momentum-conserved}
|\II \sa |= |\II \sf a_0|.
\end{equation}

Throughout this paper, we use the notation $(\cdot|\cdot)$ for the Euclidean inner product in $\R^3$ and $|\cdot|$ for the associated norm. Moreover, for a Banach space $X$, $|\cdot|_{X}$ will denote its norm, and  
$B_X(u_0,r)$  the open ball or radius $r$ centered at $u_0\in X$, with respect to the topology of $X$. 

For $q\in [1,\infty]$, $L_q(D)$ will identify the classical Lebesgue spaces on a domain $D\subset \R^n$, 
 $W^s_q(D)$ the (generalized) Sobolev spaces,   and $H^s_q(D)$ the Bessel potential spaces, for $s\in\R$.
In some of the  proofs (and when the context is clear), $|\cdot|_{D}$ will be frequently used for the $L_2$-norm on $D$. 
Furthermore, we recall the following characterization of Besov spaces 
$B^s_{qp}(D)=(H^{s_0}_q(D),H^{s_1}_q(D))_{\theta,p}$ as real interpolation of Bessel potential spaces, 
and $H^s_q(D)=[H^{s_0}_q(D),H^{s_1}_q(D)]_{\theta}$, 
with $[\cdot,\cdot]_\theta$ the complex interpolation,
valid for $s_0\ne s_1\in \R$, $p,q\in [1, \infty)$, $\theta\in (0,1)$, and  where $s=(1-\theta)s_0+\theta s_1$.
We also recall  that
$B^s_{qq}(D) = W^s_q(D)$  and $B^s_{22}(D) = W^s_2(D) =H^s_2(D)$. 

If $k\in \N$, $J\subset \R$ is an open interval and $X$ is a Banach space, then $C^k(J;X)$ (resp. $L_p(J;X)$  and $H^k_p(J;X)$) denote the space of all $k$-times continuously differentiable (resp. $L_p$- and $H^k_p$-) functions on $J$ with values in $X$.  For $1<q<\infty$, we denote by
$$L_{q,\sigma}(\Omega)=\{v\in L_q(\Omega): {\rm div}\, v=0\;\;\text{in}\;\;\Omega,\quad  (v|\nu)=0\;\;\text{on}\;\;\partial \Omega\}$$ the space of all solenoidal vector fields on $\Omega$ with zero normal component on $\partial \Omega$, and 
by  $\PP$  the Helmholtz projection of $L_q(\Omega)$ onto $L_{q,\sigma}(\Omega)$.
Finally, 
$${_0H}^s_{q,\sigma}(\Omega):=\{v\in H^s_q(\Omega)\cap L_{q,\sigma}(\Omega): v=0\;\;\text{on} \;\;\partial \Omega\}\quad\text{for } s>1/q.$$ 

%%%%%%
\section{Energy and equilibria}\label{se:energy and equilibria}
%%%%%%
In this section we provide a characterization of the equilibrium configurations.
We also show that
the critical points of the energy functional with prescribed total momentum are precisely the equilibria of the system.
Moreover, we show that if the energy with prescribed nonzero total momentum has a local minimum at 
a critical point $(0,\sa_*)$, then necessarily $\lambda_*=\max\{\lambda_1,\lambda_2,\lambda_3\}$,
corresponding  to the situation of  a stable equilibrium, see  Theorem~\ref{EV-problem}. 
\subsection{Dissipation of energy}
The available energy of system \eqref{problem} is given by
\begin{equation}
\label{energy}
{\sf E}={\sf E}(v,\sa):=\frac{1}{2}\big[ |v|_{L_2(\Omega)}^2 - (\II\omega |\omega) + (\II\sa|\sa)\big].
\end{equation}
The above functional is positive definite along the solutions 
to \eqref{problem} thanks to the following  result.

%%%%%%%%
\begin{lemma}\label{lem:positive_def}
There exists a constant $c\in (0,1]$ such that 
\[
c|v|^2_{L_2(\Omega)}\le |v|^2_{L_2(\Omega)} - (\II\omega |\omega) \le |v|^2_{L_2(\Omega)},
\quad v\in L_2(\Omega).
\] 
\end{lemma}
%%%%%%%%%%
\begin{proof}
We refer to  \cite[Sections  7.2.2--7.2.4]{KoKr00},
see also  \cite[Lemma 2.3.3 and the following remarks]{Ma16}. 
\end{proof}
Sufficiently smooth solutions to \eqref{problem} enjoy the following energy balance. 
\begin{lemma}\label{lem:energy_balance}
Consider $(v,\sa,p)$, with 
$v\in H^1_2((0,T);L_{2,\sigma}(\Omega))\cap L_2((0,T);   {_0H}^2_{2,\sigma}(\Omega))$, 
$\sa\in H^1_\infty((0,T))$ and $p\in L_2((0,T);H^1_2(\Omega))$,
 satisfying \eqref{problem} a.e. in $\Omega\times (0,T)$ for some $T\in (0,\infty]$. Then,  
$$ \frac{d}{dt}{\sf E}=-\upnu |\nabla v|^2_{L_2(\Omega)}\qquad \text{in }(0,T).$$
\end{lemma}

\begin{proof}
See for instance \cite [page 495]{DGMZ16}. 
For the reader's convenience, we will include the short proof here. 
From \eqref{energy} and \eqref{problem}, one has 
\begin{equation*}
\begin{aligned}
\frac{d}{dt}{\sf E}
& =(\partial_t v|v)_\Omega -(\II\dot\omega|\omega) + (\II\,\dot\sa |\sa) \\
& =\int_\Omega \big( (\upnu\Delta v-\nabla p|v) -(v|\nabla (v\otimes v)) 
    + ((\dot\omega -\dot\sa)  \times x|v) + 2 ((\omega -\sa)\times v|v)\big)\,dx\\
&\quad -(\II\dot\omega|\omega) +(\II\dot\sa|\sa)\\
& = -\upnu |\nabla v|^2_\Omega +(\II(\dot\omega -\dot\sa)|\omega) -(\II\dot\omega|\omega) +(\II\dot\sa|\sa)\\
& = -\upnu |\nabla v|^2_\Omega +(\II\dot\sa|\sa -\omega)
= -\upnu |\nabla v|^2_\Omega - ((\sa-\omega)\times \II\sa|\sa -\omega)
=-\upnu |\nabla v|^2_\Omega. 
\end{aligned}
\end{equation*}
\end{proof}
Next we show that the  energy is a {\em strict Lyapunov functional}, which means that the function 
$[t\mapsto {\sf E}(v(t),\sa(t))]$ is strictly decreasing along non-constant solutions.  
Suppose $\frac{d}{dt} {\sf E}=0$ on some interval $(t_1,t_2)$.
Then $\nabla v=0$ on $\Omega\times (t_1,t_2)$ and so $v=0$ on $\Omega\times (t_1,t_2)$ 
by Poincar\'e's inequality. This implies $\omega=0$ on $(t_1,t_2)$, and hence 
$$\dot\sa \times x +\nabla p=0\quad\text{on}\quad (t_1,t_2).$$
Taking the curl on both sides yields
${\rm curl}\, (\dot\sa\times x)=2\dot\sa=0,$
and hence we are at an equilibrium. 
%%%%%%%%%%%
\subsection{Equilibria}
%%%%%%%%%%%%%
 From \eqref{problem}, it can be shown that equilibrium configurations are characterized by having $v= 0$ and $\sa\in \R^3$ satisfying $\sa\times\II \sa=0$. 
Note that 
the latter implies that either $\sa = 0$ or $\sa$ and $\II\sa$ must be parallel, and therefore
$\sa\in {\sf N}(\lambda - \II)$ for some $\lambda\in\{\lambda_1,\lambda_2,\lambda_3\}.$ 
This yields  the set of nontrivial equilibria 
\begin{equation}
\label{equilibria1}
\cE= \{(0,\sa_\ast) : a_\ast\in {\sf N}(\lambda_\ast - \II),\; \lambda_\ast\in\{\lambda_1,\lambda_2,\lambda_3\}\},
\end{equation}
with constant pressure $p$ in each case.
There are three distinguished cases:
\begin{equation}
\label{equilibria2}
\begin{aligned}
\vspace{2mm}
{\bf (i)}   &\;\; \lambda_1=\lambda_2=\lambda_3:\; \cE=\{0\}\times \R^3,\;\; 
\\
{\bf (ii)}  &\;\; \lambda_i\neq \lambda_j=\lambda_k: \;
\cE=\bigcup\limits_{\ell=i,j}\{0\}\times {\sf N}(\lambda_\ell - \II), \;\; 
\\
{\bf (iii)} & \;\; \lambda_1\neq \lambda_2\neq\lambda_3:\; \cE=\bigcup\limits_{j=1}^3 \{0\}\times {\sf N}(\lambda_j -\II), \;\; 
\end{aligned}
\end{equation}
%%%%%%%%%%%
\subsection{Equilibria as critical points of the  energy}
%%%%%%%%%%
In this subsection we determine the critical points $e=(v_*,\sa_*)$ of the energy ${\sf E}(v,\sa)$ under the constraint 
%of given total angular momentum 
that the total angular momentum is conserved at all times 
$${\sf M(\sf a)}:=\frac{1}{2} (\II \sa| \II \sa)={\sf M}_0.$$
This provides a motivation for the stability conditions in Theorem~\ref{th:stability-instability}  by a variational analysis argument.
By the method of Langrange multipliers, there is $\mu\in \R$ such that at a critical point  $(v_*,\sa_*)$ we have
\begin{equation}
\label{Lagrange}
{\sf E}^\prime (v_*,\sa_\ast) + \mu {\sf M}^\prime (\sa_*)=0.
\end{equation}
The above Gateaux derivatives are given by 
\begin{equation*}
\begin{aligned}
&\langle {\sf E}^\prime (v,\sa)|({\sf w},{\sf b})\rangle 
=\int_\Omega (v|{\sf w})\,dx  
-\Big(\int_\Omega (x\times v)\,dx\Big| \II^{-1} \int_\Omega (x\times {\sf w})\,dx\Big) + (\II \sa |{\sf b} ), \\
&\langle {\sf M}^\prime  \sa | {\sf b}\rangle=(\II \sa | \II {\sf b}).
\end{aligned}
\end{equation*}
Choosing $({\sf w},{\sf b})=(v_*,0)$ in the Lagrange condition \eqref{Lagrange} yields
$$
\int_\Omega (v_*|v_*)\,dx  
-\Big(\int_\Omega (x\times v_*)\,dx\Big| \II^{-1} \int_\Omega (x\times v_*)\,dx\Big)=0,
$$
and  Lemma \ref{lem:positive_def} then implies $v_*=0$.
Setting ${\sf w}=0$ and varying ${\sf b}$ yields
$$((I +\mu\II) \sa_*|\II {\sf b})=0,\quad {\sf b}\in\R^3, $$
implying that either $\sa_*=0$, or $\sa_*$ is an eigenvector of $\II$ and $\mu=-1/\lambda_*$ an eigenvalue.
This shows that the critical points of ${\sf E}$ with prescribed total angular momentum are precisely the equilibria 
of \eqref{problem}.
\medskip\\
Suppose now that $e_*=(0,\sa_*)$ is a local minimum of the energy ${\sf E}$ with respect to the constraint 
${\sf M}={\sf M}_0\ne 0$, and $\mu=-1/\lambda_*$ is the corresponding Lagrange multiplier.
Then ${\sf C}:={\sf C}(e_*):=[{\sf E}^{\prime\prime}(e_*) + \mu {\sf M}^{\prime \prime}(e_*)]$ is positive semi-definite on the
kernel of ${\sf M}^{\prime}(\sa_*)$.
We note on the go that the kernel of ${\sf M}^{\prime}(\sa_*)$  coincides with
${\sf N}(\lambda_* -\II)^\perp,$ the orthogonal complement of the kernel of $(\lambda_*-\II)$.  
We have
\begin{equation*}
\langle {\sf C}({\sf w},{\sf b})| ({\sf w},{\sf b})\rangle
=|{\sf w}|^2_\Omega -\Big(\int_\Omega (x\times {\sf w})\,dx\Big| \II^{-1} \int_\Omega (x\times {\sf w})\,dx\Big) 
+ \frac{1}{\lambda_*}((\lambda_*-\II){\sf b}| \II{\sf b}),
\end{equation*}
By Lemma \ref{lem:positive_def},
\begin{equation}
\label{pos-def}
\langle {\sf C}({\sf w},{\sf b})| ({\sf w},{\sf b})\rangle \ge c |{\sf w}|^2_\Omega + \lambda_*^{-1} \varphi({\sf b},{\sf b}),
\end{equation}
with a constant $c\in (0,1)$, where $\varphi :\R^3\times \R^3\to\R$ is the quadratic form defined by
$\varphi({\sf c}|{\sf c}):=((\lambda_*-\II){\sf c}| \II{\sf c})$.
One readily verifies that
$\varphi|_{{\sf N}(\lambda_* -\II)^\perp \times {\sf N}(\lambda_* -\II)^\perp}=\varphi $.
From 
\eqref{pos-def} then follows that
${\sf C}$ is positive definite on the kernel of ${\sf M}^{\prime}(\sa_*)$  if and only if 
$$\lambda_*=\max\{\lambda_1,\lambda_2,\lambda_3\}.$$ 
Summarizing, we have shown the following result.
%%%%%%%%%%
\begin{proposition}\label{prop:equilibria}
The following assertions hold for problem \eqref{problem}.
\begin{enumerate}
\setlength\itemsep{1mm}
\item[(a)] The total angular momentum is conserved.
%\vspace{1mm}
\item[(b)] The  energy {\sf E}, defined in \eqref{energy}, is a strict Lyapunov functional.
%\vspace{1mm}
\item[(c)] The set of nontrivial equilibria is given by \eqref{equilibria1}--\eqref{equilibria2}.
%\vspace{1mm}
\item[(d)] The critical points of the energy with prescribed total momentum are precisely the equilibria of the system.
%\vspace{1mm}
\item[(e)] If the energy with prescribed nonzero total momentum has a local minimum at a critical point $(0,\sa_*)$
then necessarily $\lambda_*=\max\{\lambda_1,\lambda_2,\lambda_3\}$.
\end{enumerate}
\end{proposition}

%%%%%%%%%%%%%%%
\section{Local well-posedness and critical spaces}\label{se:local well-posedness and critical spaces}
%%%%%%%%%%%%
\noindent
In this section, we show that system \eqref{problem} is locally well-posed in an $L_q$-setting.
We consider the Banach spaces 
\begin{equation}
\label{X0-X1}
X_0:= L_{q,\sigma}(\Omega)\times \R^3,
\quad X_1:= {_0H}^2_{q,\sigma}(\Omega) \times \R^3,\quad 1<q<\infty.
\end{equation}
 Notice that $X_1$ is compactly embedded in $X_0$. 
Clearly,   \eqref{problem} can be rewritten in the following equivalent form 
\begin{equation}
\label{problem-2}
\begin{aligned}
\partial_t v + (\dot\sa -\dot\omega)\times x 
-\upnu \Delta v +\nabla p &=f(v,\sa) &&\text{in} && \Omega, \\
{\rm div}\, v &= 0 &&\text{in} && \Omega, \\
v &=0 &&\text{on} && \partial\Omega, \\
\II\, \dot\sa &=g(v,\sa) &&\text{on}&& \R^3,\\
(v(0),\sa(0)) &= (v_0,\sa_0) \\
\end{aligned}
\end{equation}
where  $f(v,\sa)=-v\cdot\nabla v-2(\sa-\omega)\times v$ and $g(v,\sa)=-(\sa -\omega)\times \II \sa.$ 
Let 
\begin{equation}
\label{def-E-A_0}
\begin{aligned}
E(v,\sa) &:=\big(v+\PP\big(x\times \II^{-1}\int_\Omega(x\times v)\,dx) - \PP (x\times \sa), \II\sa\big), \\
A(v,\sa) &:=\big(-\upnu\PP\Delta v , 0\big)
\end{aligned}
\end{equation}
for $u:=(v,\sa)\in {_0}H^2_{q,\sigma}(\Omega)\times\R^3$.
Then problem~\eqref{problem-2} can be formulated as a semilinear evolution equation
\begin{equation}\label{eq:evolution}
\frac{d}{dt}u +  Lu=F(u),\quad u(0)=u_0,  
\end{equation}
where $u_0:=(v_0,\sa_0)$, $L:=E^{-1}A$ and $F(u)=(F_1(u), F_2(u)):=E^{-1}(\PP f(v,\sa),g(v,\sa)),$
provided we know that $E$ is invertible.
Indeed, we have the following result.
%%%%%%%%%%
\begin{proposition}
\label{pro:E-invertible}
 $E$ is invertible on $X_0$. The inverse is given by
 \begin{equation}\label{eq:E-1}
 E^{-1}=\left[\begin{array}{cc}
I+ C & (I+ C) \PP(x\times \II^{-1}\cdot )\\
      0 & \II^{-1}   
\end{array}\right],
 \end{equation}
where  $I$ is the identity operator on $L_{q,\sigma}(\Omega)$, and where
$C$ has the following properties:
\begin{enumerate}
\setlength\itemsep{1mm}
\item[{\bf (a)}]
 C is a compact (in fact a finite-rank) operator on $L_{q,\sigma}(\Omega)$. 
%\vspace{1mm}
 \item[{\bf (b)}]
 $I+C$ is invertible on $L_{q,\sigma}(\Omega)$ and  positive definite on $L_{2,\sigma}(\Omega)$. 
\end{enumerate} 
 
\end{proposition}
%%%%%%%%%
\begin{proof}
We first observe that 
$Kv:=\PP\big(x\times \II^{-1}\int_\Omega(x\times v)\,dx)$ defines a compact linear operator on $L_{q,\sigma}(\Omega)$.
Indeed, one readily verifies that
\begin{equation}
\label{finite-rank}
Kv=\PP\big(x\times \II^{-1}\int_\Omega(x\times v)\,dx)=\sum_{i=1}^3 \langle \ell_i| v\rangle \PP (x\times {\sf e}_i),
\end{equation}
where $\ell_i$ are bounded linear functionals on $L_{q,\sigma}(\Omega)$.  
This shows that $K$ has finite rank.
We claim that $(I+K)$ is invertible on $L_{q,\sigma}(\Omega)$.
As $(I+K)$ is a compact perturbation of the identity operator,
by Fredholm theory we are done if we can show that $(I+K)$ is injective.
  Suppose  $(I+K)v=0$ for some $v\in L_{q,\sigma}(\Omega)$.
The equation $$v+\PP\big(x\times \II^{-1}\int_\Omega(x\times v)\,dx)=0$$ 
shows that $v$ is smooth, and therefore lies in $ L_{2,\sigma}(\Omega)$.
Multiplying the above equation with $v$ and integrating over $\Omega$ yields
\begin{equation*}
0=((I+K)v|v)_\Omega = |v|^2_\Omega + (\PP (x\times \omega ) | v)_\Omega = |v|^2_\Omega + (x\times \omega |v)_\Omega
=  |v|^2_\Omega - (\II \omega |\omega).
\end{equation*}
By Lemma \ref{lem:positive_def}, $v=0$, and hence the claim is proved.

Is is now easy to see that the equation $E(v,\sa)=(f,{\sf b})$ has for each $(f,{\sf b})\in X_0$ a unique solution 
given by $(v,\sa)=((I+K)^{-1}(f +\PP(x\times \II^{-1}{\sf b}), \II^{-1}{\sf b})$.

We now set $I+C:=(I+K)^{-1}$. Clearly, $I+C$ is invertible, and the  identity $I+C=I - K(I+K)^{-1}$ shows that $C$ has finite rank.
Let $v,w\in L_{2,\sigma}(\Omega)$ be given. Using well-known properties of the cross-product,  \eqref{def-omega} and Lemma~\ref {lem:positive_def}
one verifies that $(Kv|w)_{L_2(\Omega)} =(v|Kw)_{L_2(\Omega)}$ 
and 
\begin{equation}
\label{I+K positive}
((I+K)v|v)_{L_2(\Omega)} 
=|v|^2_{L_2(\Omega)}-(\II\omega |\omega)\ge c|v|^2_{L_2(\Omega)},
\end{equation}
yielding the second assertion in (b). 
\end{proof}
%%%%%%%%%%
\begin{remark}
\label{remark-1}
{\rm {\bf (a)}
The proof of Proposition~\ref{pro:E-invertible} shows that $C$ is given by
\begin{equation}
\label{C-structure}
C=-K(I+K)^{-1}=\sum_{i=1}^3 \langle L_i| \cdot \rangle \PP (x\times {\sf e}_i),
\end{equation}
where $L_i$ are bounded functionals on $L_{q,\sigma}(\Omega)$.
We note that the functions $\PP(x\times {\sf e}_i)$ are smooth (in fact harmonic).
However, they do not satisfy a Dirichlet boundary condition on $\partial\Omega$.
\medskip\\
\noindent
 {\bf (b)} If follows from Proposition~\ref{pro:E-invertible} that
\begin{equation}  
\label{eq:L}
L=\left[\begin{array}{cc}
-\upnu (I+ C)\PP\Delta  &0\\
      0 & 0
      \end{array}\right].
\end{equation} }
\end{remark}
%%%%%%%%%%
It is well-known, \cite{Gi85, NoSa03}, that the Stokes operator $-\PP \Delta$ with Dirichlet boundary condition
is invertible and has a bounded $\cH^\infty$-calculus with $\cH^\infty$-angle zero on $L_{q,\sigma}(\Omega)$.
We shall show that $-\upnu (I+ C)\PP\Delta$ enjoys analogous properties.
%%%%%%%%%%%%%%
\begin{proposition}
\label{pro:L-calculus}
The operator 
\begin{equation}\label{eq:A1}
A_1:=-\upnu (I+ C)\PP\Delta\quad \text{with} \quad {\sf D}(A_1)={_0H}^2_{q,\sigma}(\Omega)
\end{equation}
is invertible and $A_1 \in \cH^\infty(L_{q,\sigma}(\Omega))$ 
with $\cH^\infty$-angle $\phi^\infty_{A_1}<\pi/2$  on $L_{q,\sigma}(\Omega)$.
Moreover, $\eta+L\in \cH^\infty(X_0)$ 
with $\cH^\infty$-angle $\phi^\infty_{L}<\pi/2$ 
for each $\eta >0$.
\end{proposition} 
%%%%%%%%%%%
\begin{proof} We may assume w.l.o.g that $\upnu=1$.
We shall first show that $A_1$ is invertible and sectorial, with sectorial angle $\phi_{A_1}<\pi/2$.
As $A_1$ has compact resolvent,
its spectrum consists entirely of eigenvalues of finite algebraic multiplicity.
Moreover, the spectrum is independent of $q$.
Suppose that $(\mu +A_1)v=0$ for some $\mu\in\C$ and $v\in {_0H}^2_{2,\sigma}(\Omega)$.
Applying $(I+C)^{-1}=(I+K)$ to this equation and then taking the $L_2(\Omega)$ inner product with $v$ results in
$
\mu((I+K)v|v)_\Omega - (\PP \Delta v|v)_\Omega = 0.
$
We conclude that $\mu\in (-\infty,\mu_1]$ with $\mu_1<0$, as both $(I+K)$ and $-\PP\Delta$ are positive definite.
In particular, $A_1$ is invertible.

Let $ {\sf A}+ {\sf B}:=-\PP\Delta -C\PP\Delta $  with domain ${_0H}^2_{q,\sigma}(\Omega)$. 
As ${\sf A}$ admits a bounded $\cH^\infty$-calculus with $\cH^\infty$-angle zero on $L_{q,\sigma}(\Omega)$ it is, in particular,
sectorial with angle $\phi_{\sf A}=0$, while $B$ is a relative compact perturbation. It follows
from \cite[Lemma 3.1.7 and Corollary 3.1.6]{PrSi16} that there is a positive number $\eta_0$ such that $\eta_0 +{\sf A} +{\sf B}$ is invertible and  sectorial with spectral angle less than $\pi/2$.
Combining this result with the fact that $\sigma(-A_1)\subset (-\infty,\mu_1]$ we infer that 
$A_1$ is sectorial with sectorial angle $\phi_{A_1}<\pi/2$ as well.

Next we infer from Remark~\ref{remark-1}(a) that 
${\sf B}$ maps ${_0H}^2_{q,\sigma}(\Omega)$ into  $C^\infty(\bar\Omega)$. 
In particular, there is $s\in (0,1/q)$ such that ${\sf B}\in \cB({_0H}^2_{q,\sigma}(\Omega),H^{s}_{q,\sigma}(\Omega)).$
Note that for $s\in (0,1/q)$ we have
$H^s_{q,\sigma}(\Omega)=[L_{q,\sigma}(\Omega), {_0H}^2_{q,\sigma}(\Omega)]_{s/2} ={\sf D} (A^{s/2})$. 
Hence $${\sf B}:{\sf D(A)}\to  {\sf D}({\sf A}^\alpha),$$ with $\alpha=s/2$, is bounded.
Observing that ${\sf A}+{\sf B}$ is invertible and sectorial, 
we can now follow the proof of \cite[Proposition 3.3.9]{PrSi16} to infer that 
${\sf A} +{\sf B}\in \cH^\infty(L_{q,\sigma}(\Omega))$ 
with $ \cH^\infty$-angle $\phi^\infty_{{\sf A} +{\sf B}}<\pi/2$.

The assertion for $\eta +L={\rm diag}\,[\eta+A_1, \eta ]$ is easy to verify, 
see for instance \cite[Proposition 3.3.14(iv)]{PrSi16} for  $\eta +A_1\in \cH^{\infty}(L_{q,\sigma}(\Omega))$.
\end{proof}
 We are now ready to state and prove a well-posedness result. 
 Well-posedness will be established in the following {\em time-weighted spaces}
\begin{equation}
\label{E1-mu}
\EE_{1,\mu}(0,T):=H^1_{p,\mu}((0,T);X_0)\cap L_{p,\mu}((0,T);X_1),
\end{equation}
where, for $T\in (0,\infty)$, $1/p<\mu\le 1$, and $X$ a Banach space,
\begin{equation*}
\begin{aligned}
&u\in L_{p,\mu}((0,T);X) && \Leftrightarrow \quad  t^{1-\mu}u\in L_p((0,T);X),\\
&u\in H^1_{p,\mu}((0,T);X) && \Leftrightarrow \quad u,\dot u\in L_{p,\mu}((0,T);X).
 \end{aligned}
\end{equation*}
In the following we set 
\begin{equation*}
{_0B}^{s}_{qp,\sigma}(\Omega):=
\left\{ 
\begin{aligned}
&\{u\in  B^{s}_{qp}(\Omega)\cap L_{q,\sigma}(\Omega): u=0\; \text{on}\; \partial\Omega\}, &&s>1/q,\\
&B^{s}_{qp}(\Omega)\cap L_{q,\sigma}(\Omega), && s\in [0,1/q).\\
\end{aligned}
\right.
\end{equation*}
%%%%%%%%%%%%%
\begin{theorem}\label{th:local_strong}
Suppose
\begin{equation}
\label{assumptions-pq}
p\in(1,\infty),\quad q\in (1,3),\quad  2/p +3/q\le 3,
\end{equation}
and let $\mu$ satisfy
\begin{equation}
\label{assumptions-mu}
\mu\in (1/p,1],\quad  \mu\ge \mu_{\rm crit}=\frac{1}{p} + \frac{3}{2q}-\frac{1}{2}. 
\end{equation}
\begin{enumerate}
\setlength\itemsep{1mm}
\item[{\bf (a)}]
Let  $u_0=(v_0,\sa_0)\in {_0B}^{2\mu-2/p}_{qp,\sigma}(\Omega)\times \R^3=X_{\gamma,\mu}$ be given.
Then there are positive constants $T=T(u_0)$ and $\eta=\eta(u_0)$ such that
\eqref{eq:evolution} admits a unique solution $u(\cdot, u_0)=(v,{\sf a})$ in 
 \begin{equation*}
 \qquad
 {\EE_{1,\mu}(0,T)}=H^1_{p,\mu}((0,T); L_{q,\sigma}(\Omega)\times\R^3)
 \cap L_{p,\mu}((0,T); {_0H}^2_{q,\sigma}(\Omega)\times\R^3)
 \end{equation*}
 for any initial value $u_1=(v_1,\sa_1)\in B_{X_{\gamma,\mu}}(u_0,\eta)$. 
 Furthermore, there exists a positive constant $c=c(u_0)$ such that 
\begin{equation}
\label{eq:continuous_dep}
\|u(\cdot, u_1)-u(\cdot,u_2)\|_{\EE_{1,\mu}(0,T)}\le c|u_1-u_2|_{X_{\gamma,\mu}}
\end{equation}
for all $u_i=(v_i, \sa_i)\in B_{X_{\gamma,\mu}}(u_0,\eta)$, $i=1,2$. 
%\vspace{1mm}
\item[{\bf(b)}]
Suppose $p_j, q_j$, $\mu_j$ satisfy \eqref{assumptions-pq}-\eqref{assumptions-mu}  and, in addition, $p_1\leq p_2$, $q_1\leq q_2$ as well as
\begin{equation}\label{mu-j}
 \mu_1- \frac{1}{p_1}- \frac{3}{2q_1} \ge  \mu_2- \frac{1}{p_2}- \frac{3}{2q_2}.
 \end{equation}
Then for each initial value
$(v_0,{\sf a}_0)\in {_0B}^{2\mu_1 -2/p_1}_{q_1 p_1,\sigma}(\Omega)\times \R^3,$ 
problem  \eqref{eq:evolution} admits a unique solution $(v,{\sf a})$ in the class
\begin{equation*}
\begin{split}
&H^1_{p_1,\mu_1}((0,T); L_{q_1,\sigma}(\Omega)\times\R^3)\cap L_{p_1,\mu_1}((0,T); {_0H}^2_{q_1,\sigma}(\Omega)\times\R^3) \\
&\cap H^1_{p_2,\mu_2}((0,T); L_{q_2,\sigma}(\Omega)\times\R^3)\cap L_{p_2,\mu_2}((0,T); {_0H}^2_{q_2,\sigma}(\Omega)\times\R^3).
\end{split}
\end{equation*}
\item[{\bf (c)}]
Each solution with initial value in  $(v_0,\sa_0)\in {_0B}^{2\mu-2/p}_{qp,\sigma}(\Omega)\times \R^3$
exists on a maximal interval  $[0,t_+)=[0,t_+(v_0,\sa_0))$, and 
enjoys the additional regularity property
\begin{equation*}
\qquad \qquad v \in C([0,t_+); {_0B}^{2\mu-2/p}_{qp,\sigma}(\Omega))\cap C((0,t_+);{_0B}^{2-2/p}_{qp,\sigma}(\Omega)),
\;\; \sa\in C^1([0,t_+),\R^3).
\end{equation*}
Hence, $v$ regularizes instantaneously if $\mu<1$.
%\vspace{1mm}
\item[{\bf (d)}]
The solution $u=(v,{\sa})$ exists globally if
 $u([0,t_+))\subset B^{2\mu-2/p}_{qp}(\Omega)\times \R^3$ is relatively compact. 
\end{enumerate}
\end{theorem}
%%%%%%%
\begin{proof}
The assertions in (a),(c),(d) follow from \cite[Theorem 2.1]{PrSiWi18} (or \cite[Theorem~1.2]{PrWi17}) provided 
we can verify the assumptions therein.

For $\beta\in (0,1)$, let $X_\beta:=[X_0,X_1]_\beta$.
Then we have
$X_\beta={_0H^{2\beta}_{q,\sigma}(\Omega)}\times \R^3$, where
${_0H}^{s}_{q,\sigma}(\Omega)$ is defined by
\begin{equation}
\label{H2-interpolation}
{_0H}^{s}_{q,\sigma}(\Omega):=
\left\{ 
\begin{aligned}
&\{u\in  H^{s}_{q}(\Omega)\cap L_{q,\sigma}(\Omega): u=0\; \text{on}\; \partial\Omega\}, &&s>1/q,\\
&H^{s}_{q}(\Omega)\cap L_{q,\sigma}(\Omega), && s\in [0,1/q).\\
\end{aligned}
\right.
\end{equation}
We note that the nonlinearity $F$ has a bilinear structure, i.e., $F(u)=G(u,u)$, where
\begin{equation}\label{eq:bilinear_nonlinearity}
G(u_1,u_2)=-E^{-1}(\PP(v_1\nabla v_2  + 2(\sa_1-\omega_1)\times v_2), (\sa_1-\omega_1)\times \II\sa_2)
\end{equation} 
and $u_i=(v_i,\sa_i)$.
We  claim that 
$G:X_\beta\times X_\beta\to X_0$ is  bounded for  $\beta=\frac{1}{4}\big(1+\frac{3}{q}\big).$
As in  \cite[Section 3]{PrWi17} we obtain boundedness of 
$[(v_1,v_2)\mapsto v_1\nabla v_2]: X^1_\beta\times X^1_\beta\to X^1_0$,
where we use the canonical decomposition $X_\theta=X^1_\theta \times\R^3$.
The remaining term  $ (2(\sa_1-\omega_1)\times v_2, (\sa_1-\omega_1)\times \II\sa_2)$ is easy to estimate,
and the claim follows from Proposition~\ref{pro:E-invertible} and boundedness of $\PP$ in $L_q(\Omega)$.
It follows that the assumptions of  \cite[Theorem 3.1]{PrSiWi18} are satisfied for all $p$ and $q$ in the range indicated above.
(Alternatively, one verifies that  (H1)-(H2) in \cite{PrWi17} holds with $\rho_j=1$ and $\beta_j=\beta$).
 
It remains to address the assumption  $L\in\mathcal{BIP}(X_0)$ imposed in \cite[Section~2]{PrSiWi18},
(or, alternatively, the structural condition (S) in \cite{PrWi17}).
In either case, the assumption is used to ensure that the mixed-derivative embedding 
\begin{equation}
\label{mixed-derivative}
{_0\EE}_{1,\mu}(0,a)\hookrightarrow {_0H}^{1-\beta}_{p,\mu}((0,a);X_\beta)
\end{equation}
holds, where ${_0\EE}_{1,\mu}(0,a)=\{u\in \EE_{1,\mu}(0,a):\; u(0)=0\}.$ 
We observe that \eqref{mixed-derivative} follows from \cite[Remark 1.1]{PrWi17} and Proposition~\ref{pro:L-calculus} 
by choosing $A_{\#}=\eta+L$ for $\eta>0$. 
It is worthwhile to mention that in our particular situation it actually suffices to have
$A_1\in \mathcal{BIP}(L_{q,\sigma}(\Omega))$  for $L={\rm diag}\,[A_1,0]$,
as the embedding \eqref{mixed-derivative} is automatically satisfied for 
${_0H}^1_{p,\mu}((0,a);\R^3)$.
As a consequence, all results of Section 2 of \cite{PrSiWi18} hold true for the semilinear evolution equation~\eqref{eq:evolution}. 

We are left with proving (b). We note that the embedding $B^s_{q_1p_1}(\Omega)\subset B^s_{q_1p_2}(\Omega)$
and Sobolev embedding  yields
$${B}^{2\mu_1-2/p_1}_{q_1p_1}(\Omega)\subset {B}^{2\mu_2-2/p_2}_{q_2p_2}(\Omega).$$
Therefore, the initial value $(v_0,{\sf a}_0)$ also belongs to the second space in the line above.
Let $(v_j,{\sf a}_j)$ be the unique solution of \eqref{eq:evolution} in the respective class
\begin{equation}\label{EE-j}
\EE_{1,\mu_j}(0,T):=H^1_{p_j,\mu_j}((0,T); L_{q_j,\sigma}(\Omega)\times\R^3)\cap L_{p_j,\mu_j}((0,T); {_0H}^2_{q_j,\sigma}(\Omega)\times\R^3).
\end{equation} 
In both cases, the solutions were obtained as a fixed point of the strict contraction
$$ {\sf T}: {\sf M_j}\to {\sf M_j},\quad  {\sf T}(v,\sa) :=e^{-tL}(v_0, {\sf a}_0) + e^{-tL}\star F(v,{\sf a}), $$
where ${\sf M_j}$ is a closed subset of $\EE_{\mu_j}(0,T)$, respectively.
But ${\sf T}: {\sf M_1}\cap {\sf M_2}\to {\sf M_1}\to {\sf M_2}$ is also a strict contraction,
and thus has a unique fixed point $(v_\star, a_\star)\in {\sf M_1}\cap {\sf M_2}$.
Therefore,  $(v_1,{\sf a}_1)=(v_2,{\sf a_2})=(v_\star, a_\star)$
and the assertion in (b) follows.
\end{proof} 
%%%%%%%
\begin{remarks}\label{remark-2}
{\rm {\bf (a)}
For the case  $\mu=\mu_{\rm crit}$ we get the critical spaces
\begin{equation}
\label{strong-critical-space}
X_{\rm crit}={_0B}^{3/q-1}_{pq,\sigma}(\Omega)\times \R^3.
\end{equation}
Theorem~\ref{th:local_strong} then shows that equation \eqref{problem-2} admits a unique maximal solution
for initial values $(v_0,\sa_0)$  in  $X_{\rm crit}$.
Here one should observe that the homogeneous space $\dot B^{3/q-1}_{qp}(\R^3)$
is  scaling invariant for the Navier-Stokes equations on $\R^3$.
It is also interesting to note that each critical space ${_0B}^{3/q-1}_{pq,\sigma}(\Omega)$ has  Sobolev index $-1$,
independently of $p$ and $q$.
\medskip\\
\noindent
{\bf (b)}
The case $p_1=q_1=2$ is admissible in Theorem~\ref{th:local_strong}, and we then obtain $\mu_{\rm crit}=3/4$ and 
\begin{equation}
\label{p=q=2}
X_{\rm crit}=(L_{2,\sigma}(\Omega)\times\R^3, {_0H}^2_{2,\sigma}(\Omega)\times\R^3)_{1/4,2}
\subset {H}^{1/2}_{2,\sigma}(\Omega)\times \R^3.
\end{equation}
Hence, Theorem~\ref{th:local_strong}(b) asserts that  \eqref{eq:evolution} has for each initial value $(v_0,{\sf a}_0)\in X_{\rm crit}$ a unique solution in the class
\begin{equation*}
\begin{split} 
&H^1_{2,3/4}((0,T); L_{2,\sigma}(\Omega)\times\R^3)\cap L_{2,3/4}((0,T); {_0H}^2_{2,\sigma}(\Omega)\times\R^3) \\
&\cap H^1_{p,\mu}((0,T); L_{q,\sigma}(\Omega)\times\R^3)\cap L_{p,\mu_2}((0,T); {_0H}^2_{q,\sigma}(\Omega)\times\R^3),
\end{split}
\end{equation*}
for any $p\ge 2, q\in [2,3)$, with $\mu=1/p + 3/2q -1/2.$
In particular, we can conclude that $v\in C((0,t_+); B^{2-2/p}_{qp}(\Omega))$ for any $p\ge 2, q\in [2,3)$.

The result for the case $p=q=2$ is reminiscent of the celebrated Fujita-Kato Theorem \cite{FuKa62} for the Navier-Stokes equations.
\smallskip\\
\noindent
{\bf (c)} Theorem~\ref{th:local_strong}(b) asserts that problem \eqref{eq:evolution} admits for each initial value
$$(v_0,{\sf a_0})\in {_0H}^1_{2,\sigma}(\Omega)\times\R^3$$
a unique solution in the class 
\begin{equation*}
\begin{split}
&H^1_{2}((0,T); L_{2,\sigma}(\Omega)\times\R^3)\cap L_{2}((0,T); {_0H}^2_{2,\sigma}(\Omega)\times\R^3) \\
&\cap H^1_{p,\mu}((0,T); L_{q,\sigma}(\Omega)\times\R^3)\cap L_{p,\mu}((0,T); {_0H}^2_{q,\sigma}(\Omega)\times\R^3),
\end{split}
\end{equation*}
for any $p\ge 2, q\in [2,3)$, with $\mu=1/p +3/2q-1/4$.
In particular, we can conclude that $v\in C((0,t_+); B^{2-2/p}_{qp}(\Omega))$ for any $p\ge 2, q\in [2,3)$.
}
\end{remarks}  

%%%%%%%%%%%

\section{The spectrum of $L_*$}\label{se:spectrum}
 In this section, we will investigate the linear stability of equilibria for \eqref{problem}. 

\medskip
Suppose $(0,\sa_*)$ is a nontrivial equilibrium of \eqref{problem}, i.e. $\sa_*\in \R^3\setminus\{0\}$ and 
\begin{equation}
\II\sa_*=\lambda_* \sa_*\quad\text{for some}\quad \lambda_*\in\{\lambda_1,\lambda_2,\lambda_3\}.
\end{equation}
The linearization of  \eqref{problem} at an equilibrium $(0,\sa_*)$  is given by 
\begin{equation}
\label{linearize}
\begin{aligned}
\partial_t v + (\dot\sa -\dot\omega)\times x + 2({\sa_*}\times v) 
-\upnu \Delta v +\nabla p &=f &&\text{in} && \Omega, \\
{\rm div}\, v &= 0 &&\text{in} && \Omega, \\
v &=0 &&\text{on} && \partial\Omega, \\
\II\, \dot\sa + \sa_*\times \II\sa + (\sa-\omega)\times\II\sa_* &=g &&\text{on}&& \R^3,\\
(v(0),\sa(0)) &= (v_0,\sa_0) \\
\end{aligned}
\end{equation}
Let $A_*$  and $L_*$ be defined by
\begin{equation}
\label{def-A_*}
\begin{aligned}
A_*(v,\sa) &:=\!\big(\!-\!\upnu\PP\Delta v  + 2 \PP(\sa_*\times v), 
                \lambda_*\sa_*\times \II^{-1}\! \int_\Omega \!\!(x\times v)\,dx + \sa_* \times (\II-\lambda_*)\sa\big) \\
L_*(v,\sa)&:= E^{-1} A_*(v,\sa)       
\end{aligned}
\end{equation}
for $(v,\sa)\in {_0}H^2_{q,\sigma}(\Omega)\times\R^3$,
where  $E$ has been introduced in \eqref{def-E-A_0}.
Then the linear problem~\eqref{linearize} is equivalent to
\begin{equation}\label{eq:evolution-2}
\frac{d}{dt}u + L_*u=F_*(t),\quad u(0)=u_0,  
\end{equation}
where  $u:=(v,\sa)$, $u_0:=(v_0,\sa_0)$, and  $F_*(t):=E^{-1}(\PP f(t),g(t)).$
We can prove the following result

%%%%%%%%%%%%%%
\begin{proposition}
\label{pro:L*-calculus}
There exists a number $\eta_0>0$ such that $\eta+L_*\in \cH^\infty(X_0)$ 
with $\cH^\infty$-angle $\phi^\infty_{L_*}<\pi/2$ 
for each $\eta\ge \eta_0$.
\end{proposition} 
%%%%%%%%%%%
\begin{proof}
A short computation shows that 
\begin{equation}  \label{eq:L*}
L_*=\left[\begin{array}{cc}
-\upnu (I+ C)\PP\Delta  &0\\
      0 & I_{\R^3}
      \end{array}\right]+R, 
\end{equation} 
where
$R$ is a bounded operator on $X_0$.
The assertion follows now from Proposition~\ref{pro:L-calculus} and \cite[Corollary 3.3.15]{PrSi16}.
\end{proof}

In the following we assume that the eigenvalues $\lambda_j$ of $\II$ are ordered by $\lambda_1\le \lambda_2\le\lambda_3$.
We note that the eigenvalue problem $zu+L_*u=0$ is equivalent to $zEu+A_*u=0$, which in turn reads
\begin{equation}
\label{EV-problem}
\begin{aligned}
z\big( v + \PP((\sa -\omega)\times x)\big) + 2\PP({\sa_*}\times v) - \upnu\PP \Delta v &=0 &&\text{in} && \Omega, \\
v &=0 &&\text{on} && \partial\Omega, \\
z\II\sa + \sa_*\times \II\sa + (\sa-\omega)\times\II\sa_* &=0 &&\text{on}&& \R^3.\\
\end{aligned}
\end{equation}
%%%%%%
\begin{theorem} 
\label{th:spectrum}
Let $(0,\sa_*)\in \cE$ be an equilibrium of system~\eqref{problem} with $\sa_*\neq 0$,
and recall that $\sa_*\in {\sf N}(\lambda_*-\II)$.
Then we have the following properties.
\begin{enumerate}
\setlength\itemsep{1mm}
\item[({\bf a)}] The spectrum of $L_*$ consists of countably many eigenvalues which are all of finite 
algebraic multiplicity.
%\vspace{1mm}
\item[{\bf (b)}] $0$ is a semi-simple eigenvalue with multiplicity equal to the dimension of $\cE$.
%\vspace{-3mm}
\item[{\bf (c)}] $\sigma(L_*)\setminus\{0\}\cap i\R=\emptyset.$
%\vspace{1mm}
\item[{\bf (d)}]  All eigenvalues of $\sigma(-L_*)\setminus\{0\}$ have negative real part if 
$\lambda_*=\lambda_3$.
%\vspace{1mm}
\item[{\bf (e)}] 
$-L_*$ has exactly one positive eigenvalue if $\lambda_1\le \lambda_\ast=\lambda_2<\lambda_3$.
%\vspace{1mm} 
\item[{\bf (f)}]
$-L_*$ has exactly two eigenvalues in $[{\rm Re}\,z>0]$ if $\lambda_*=\lambda_1<\lambda_2\le \lambda_3$.
\end{enumerate}
In conclusion, $(0,\sa_*)$ is normally stable if $\lambda_*=\max\{\lambda_1,\lambda_2,\lambda_3\}$,
and normally hyperbolic otherwise.
\end{theorem}
%%%%%%%%%%%%%%%%%
\begin{proof}
{\bf (a)} As $L_*$ has compact resolvent,
the spectrum of $L_*$ consists entirely of eigenvalues of finite algebraic multiplicity.
\medskip\\
\noindent
{\bf(b)}
Suppose $z=0$. Then the eigenvalue problem~\eqref{EV-problem} reads
\begin{equation}
\label{EV-problem-0}
\begin{aligned}
 2\PP({\sa_*}\times v) - \upnu \PP \Delta v &=0 &&\text{in} && \Omega, \\
v &=0 &&\text{on} && \partial\Omega, \\
\sa_*\times \II\sa + (\sa-\omega)\times\II\sa_* &=0 &&\text{on}&& \R^3.\\
\end{aligned}
\end{equation}
Multiplying the first equation with $v$ and integrating over $\Omega$ yields
\begin{equation*}
\begin{aligned}
0=2 (\PP({\sa_*}\times v)|v)_\Omega -\upnu (\PP \Delta v|v)_\Omega
  =2({\sa_*}\times v|v)_\Omega  -\upnu (\Delta v|v)_\Omega = \upnu |\nabla v|^2_\Omega,
\end{aligned}
\end{equation*}
where we used that $v \perp ({\sa_*}\times v )$.
From the Dirichlet boundary condition follows $v=0$.
The third equation~\eqref{EV-problem-0} now reduces to 
$$  0= \sa_*\times \II\sa + \sa\times\II\sa_* = \sa_*\times (\II-\lambda_*)\sa.$$
Taking the cross product  with $\sa_*$ results in 
$$
0=\sa_*\times \big(\sa_*\times (\II-\lambda_*)\sa\big)
=(\sa_*|(\II-\lambda_*)\sa)\sa_* - |\sa_*|^2 (\II-\lambda_*)\sa
=- |\sa_*|^2 (\II-\lambda_*)\sa.
$$
This implies $\sa\in {\sf N}(\lambda_*-\II)$,  and we therefore have $(v,\sa)\in\cE$.

Next we show that $z=0$ is semi-simple. From \eqref{eq:L*} it follows that $L_*$ is a Fredholm operator of index 0. Hence, to prove that $z=0$ is a semi-simple eigenvalue, it suffices to verify that ${\sf N}(L_*^2)={\sf N}(L_*)$.
Suppose then that $L_*^2 \hat u=0$, and let $u:=L_*\hat u$.
Then $L_*u=0$, and hence $u=(0,\sa)$ with $\sa\in {\sf N}(\lambda_*-\II)$ by the first part.
We want to show that $\sa=0$, since then $L_*\hat u=0$, showing that $\hat u\in {\sf N}(L_*)$.
For this purpose, note that the relation $L_*(\hat v,\hat\sa)=(0,{\sf a})$  can be restated
as $A(\hat v,\hat \sa)=E(0,\sa )$, or equivalently as
\begin{equation}
\label{EV-2}
\begin{aligned}
-\upnu\PP\Delta \hat v  + 2 \PP(x\times \hat v) & = -\PP (x\times \sa) \\
 \lambda_* (\sa_*\times \hat \omega)  + \sa_* \times (\II-\lambda_*)\hat\sa & = \II \sa,
\end{aligned}
\end{equation}
where $\hat\omega=\II^{-1}\int_\Omega (x\times \hat v)\,dx$.
Multiplying the first equation in~\eqref{EV-2} by $\hat v$ and integrating over $\Omega$ yields
\begin{equation}
\label{BB}
\begin{aligned}
0 &=-\upnu (\PP\Delta \hat v| \hat v)_\Omega  + 2 (\PP(x\times \hat v) | \hat v)_\Omega + (\PP (x\times \sa) | \hat v)_\Omega\\
   &= \upnu |\nabla \hat v|^2_\Omega +2 (x\times \hat v | \hat v)_\Omega + (x\times \sa | \hat v)_\Omega\\
   &= \upnu |\nabla \hat v|^2_\Omega -(\II\sa | \hat\omega) = \upnu |\nabla \hat v|^2_\Omega -\lambda_* (\sa | \hat\omega).
\end{aligned}
\end{equation}
From the second line in \eqref{EV-2} follows $0=(\II \sa | \sa_*)=\lambda_*(\sa | \sa_*)$ showing that $\sa\perp \sa_*$.
Next we observe that
\begin{equation*}
\begin{aligned}
\lambda_* \sa_* \times \sa  &= \sa_* \times \II \sa 
= \sa_* \times (\sa_* \times (\lambda_*\hat\omega + (\II -\lambda_*)\hat\sa))\\
&=(\sa_* | \lambda_*\hat\omega + (\II -\lambda_*)\hat\sa) \sa_*  -| \sa_* |^2 (\lambda_*\hat\omega + (\II -\lambda_*)\hat\sa).
\end{aligned}
\end{equation*}
Taking the scalar product of this equation with $\sa$, and using that $\sa\perp \sa_*$ as well as  $\sa \in{\sf N}(\lambda_*-\II)$,  yields
\begin{equation*}
\begin{aligned}
0&=\lambda_* ( \sa_* \times \sa| \sa) 
=(\sa_* | \lambda_*\hat\omega + (\II -\lambda_*)\hat\sa) (\sa_* | \sa)  
-| \sa_* |^2 (\lambda_*\hat\omega + (\II -\lambda_*)\hat\sa)| \sa)\\
&=-\lambda_*| \sa_* |^2 ( \hat\omega | \sa).
\end{aligned}
\end{equation*}
Hence, $(\hat\omega | \sa)=0$, as  $\sa_*\neq 0$ by assumption. 
We conclude from~\eqref{BB} and Poincar\'e's inequality that $\hat v=0$, which in turn also implies $\hat\omega=0$. 
Therefore, 
$$\II \sa = \sa_* \times (\II-\lambda_*)\hat\sa \quad\text{and}\quad \PP (x\times \sa)=0,$$
which implies $x\times a=\nabla p$ for some function $p$.
Applying ${\rm curl }$ to this equation gives
\begin{equation*}
\begin{aligned}
0= {\rm curl}\, (\nabla p)= {\rm curl}\,(x\times \sa)=-2\sa,
\end{aligned}
\end{equation*} 
and so $\sa=0$, which is what we wanted to show. 
\medskip\\
\noindent
{\bf(c)} 
Suppose that $z\neq 0$ is an eigenvalue of \eqref{EV-problem} with $(v,\sa)$ a corresponding eigenfunction.
Multiplying the first equation with $\bar v$ and $\bar \omega$, the complex conjugates of $v$ and $\omega$ respectively, and integrating over $\Omega$ results in
\begin{equation*}
\begin{aligned}
0=& z|v|^2_\Omega + \upnu |\nabla v|^2_\Omega + z(a-\omega | \int_\Omega (x\times \bar v)\,dx)
 + 2 (\sa_* | \int_\Omega (v\times \bar v)\,dx)
\\
&= z\big(|v|^2_\Omega- (\II \omega | \omega)_{\CC^3}\big) + \upnu |\nabla v|^2_\Omega + z (\II \sa | \omega)_{\CC^3}
-4i (\sa_* | \int_\Omega ({\rm Re}\, v \times {\rm Im}\, v)\,dx)_{\R^3}.
\end{aligned}
\end{equation*} 
Taking real parts results in
\begin{equation}
\label{Re-1}
 ({\rm Re}\,z)\big(|v|^2_\Omega- (\II \omega | \omega)_{\CC^3}\big) +{\rm Re}\,[z (\II \sa | \omega)_{\CC^3}]
 + \upnu |\nabla v|^2_\Omega =0.
\end{equation}
Next we show that
\begin{equation}
\label{Re-2}
{\rm Re}\, [z(\II a|\omega)_{\CC^3}] = \lambda_*^{-1}({\rm Re}\,z) (\II \sa |(\lambda_*-\II)\sa).
\end{equation}
In order to see this, we use the last equation of~\eqref{EV-problem}
\begin{equation}
\label{EV-last-line}
 z\II\sa + \sa_*\times \II\sa + (\sa-\omega)\times\II\sa_* =0.
 \end{equation}
Taking the inner product  of \eqref{EV-last-line} in $\CC^3$ with $\omega$ and $\sa$, 
the complex conjugates of $\sa$ and $\omega$, respectively,
yields
\begin{equation*}
\begin{aligned}
& z(\II \sa| \omega)_{\CC^3} + (\sa_* \times \II \sa|\omega)_{\CC^3} + ((\sa-\omega)\times \II \sa_*|\omega)_{\CC^3}=0,\\
& z(\II \sa| \sa )_{\CC^3} + (\sa_* \times \II \sa| \sa)_{\CC^3} + ((\sa-\omega) \times \II \sa_*|\sa)_{\CC^3}=0.
\end{aligned}
\end{equation*}
Subtracting the second line from the first one gives
\begin{equation}
\label{B}
z (\II \sa  |\omega)_{\CC^3}= z(\II \sa |\sa )_{\CC^3} + (\sa _* \times \II \sa |\sa -\omega)_{\CC^3} 
- 2i ({\rm Re}\,(\sa -\omega)\times {\rm Im}\, (\sa -\omega)|\II \sa _*). 
\end{equation}
Taking the inner product  of \eqref{EV-last-line}  in $\CC^3$ with $\II \sa$ results in 
\begin{equation}
\label{C}
z (\II \sa |\II \sa )_{\CC^3}-2i({\rm Re}\, \II\sa \times {\rm Im}\, \II\sa |\sa _*) + \lambda_* (\sa _* \times \II \sa | \overline{\sa -\omega})_{\CC^3}=0.
\end{equation}
Combining \eqref{B} and \eqref{C} implies \eqref{Re-2}.
We can now substitute \eqref{Re-2} into \eqref{Re-1} to obtain the following key identity
\begin{equation}
\label{Re-3}
 ({\rm Re}\,z)\big(|v|^2_\Omega- (\II \omega | \omega)_{\CC^3}\big) + \lambda_*^{-1}({\rm Re}\,z) (\II \sa |(\lambda_*-\II)a)
 + \upnu |\nabla v|^2_\Omega =0.
\end{equation}
If follows from Lemma \ref{lem:positive_def}  (applied to ${\rm Re}\, v$ and ${\rm Im}\, v$ separately) that
\begin{equation}
\label{Re-4}
|v|^2_\Omega- (\II \omega | \omega)_{\CC^3}\ge c |v|^2_\Omega.
\end{equation}
We are now ready to show that $\sigma(L_*)\setminus\{0\}\cap i\R=\emptyset.$
Suppose  that $z\in i \R\setminus\{0\}$ is an eigenvalue of \eqref{EV-problem} with corresponding 
eigenfunction $(v,\sa)$.
Then \eqref{Re-3} implies $\nabla v=0$, and hence $v=0$, and this  in turn yields $\omega=0$.
Hence the first line in \eqref{EV-problem} reduces to
$\PP(\sa \times x)=0$, as $z\neq 0$ by assumption.
Therefore, there is a function $p$ such that $\nabla p=\sa \times x$.
Taking the curl results in ${\rm curl}\, (\sa\times x)=-2\sa=0.$
In conclusion, $(v,\sa)=(0,0)$, a contradiction.
\medskip\\
{\bf (d)}
Suppose $\lambda_*=\lambda_3$ and $z\in\C\setminus\{0\},$ 
${\rm Re}\, z>0$, is an eigenvalue of \eqref{EV-problem} with corresponding eigenfunction $(v,\sa)$.
Then  $(\II \sa |(\lambda_*-\II)a)\ge 0$, and \eqref{Re-3}--\eqref{Re-4} 
imply once again $v=0$. The same argument as in part (c) then yields $(v,\sa)=(0,0)$,
a contradiction. 
\medskip\\
{\bf (e)} 
Let $\II={\rm diag}[\mu_1,\mu,\mu_3]$. 
In the following, we assume  that $\lambda_*=\mu_3$.
From the steps (b) and (d) above we know the following facts:
\begin{enumerate}
\setlength\itemsep{1mm}
\item [(i)] Suppose  $\mu_1<\mu<\mu_3$. Then $0$ is an eigenvalue of $-L_*$ of multiplicity one,
with eigenfunction $(0,\sa_*)$, where $\sa_*\in {\sf N}(\lambda_*-\II)={\rm span}\{e_3\}.$
All remaining eigenvalues of $\sigma(-L_*)\setminus\{0\}$ have negative real part.
\item[(ii)] Suppose $\mu_1<\mu=\mu_3$.
Then  $0$ is an eigenvalue of $-L_*$ of multiplicity two,
with eigenfunction $(0,\sa_*)$, where $\sa_*\in {\sf N}(\lambda_*-\II)={\rm span}\{e_2,e_3\}.$
All remaining eigenvalues of $\sigma(-L_*)\setminus\{0\}$ have negative real part.
\end{enumerate}
This shows that as $\mu\uparrow \mu_3$, an eigenvalue $z=z(\mu) $ of $-L_*$ moves from $[{\rm Re}\,z]<0$ to~$0$.
We will prove that this eigenvalue  crosses $0$ along the real axis with positive speed
as $\mu\uparrow \mu_3$.
This implies that $-L_*$ will have a positive eigenvalue   $z(\mu)$ is case $\mu_1<\mu_3<\mu$.

Suppose  $\mu_1<\mu,\mu_3$ and let $\sa_*=\alpha_*e_3$. (Note that $\II \sa_*=\mu_3\sa_*$). 
We will study the eigenvalue problem
\begin{equation}
\label{EV-mu}
z(\mu)\big(v(\mu), \sa(\mu)\big)+L_*(\mu) \big(v(\mu),\sa(\mu)\big)=(0,0)
\end{equation}
for $ \mu \in (\mu_3-\eps_0,\mu_3 +\eps_0)$,
where  $L_*(\mu) =E^{-1}(\mu)A(\mu)$ is defined by replacing $\II$ by $\II(\mu)={\rm diag}\,[\mu_1,\mu,\mu_3]$ 
in~\eqref{def-A_*}. 
Let 
\begin{equation}
\begin{aligned}
G(\mu, (z,v,\sa)):=\big((z+L_*(\mu))(v,\sa), P_3\sa, |\sa|^2-1\big), 
\end{aligned}
\end{equation}
where $(\mu, (z,v,\sa))\in (\mu_3-\eps_0,\mu_3 +\eps_0)\times \big(\R\times {_0H}^2_{q,\sigma}(\Omega)\times\R^3\big)$,
and where $P_3{\sf b}:={\sf b}_3$  for  ${\sf b}=({\sf b}_1,{\sf b}_2,{\sf b}_3)\in \R^3.$ 
As $L_*(\mu_3)(0,e_2)=(0,0)$ (note that $\II(\mu_3)={\rm diag}\,[\mu_1,\mu_3,\mu_3]$) we obtain
$G(\mu_3,(0,0,e_2))=(0,0,0)$
by the results established above.
We will now employ the Implicit Function Theorem.
Computing the Fr\'echet derivative $D_2G$ of $G$ with respect to the variables $(z,v,\sa)$, 
and making use of $L_*(\mu_3)(0,e_2)=(0,0)$, yields
\begin{equation}
D_2G(\mu_3, (0,0,e_2))[\hat z, \hat v, \hat \sa]
= \big(\hat z(0,e_2)+L_*(\mu_3)(\hat v,\hat \sa), P_3\, \hat\sa , 2(\hat\sa | e_2)\big).
\end{equation}
In order to verify that $D_2G(\mu_3, (0,0,e_2))$ is an isomorphism, if suffices to show that
it is injective, as $L_*(\mu_3)$ is a Fredholm operator of index zero.
Assume then that
$$D_2G(\mu_3, (0,0,e_2))[\hat z, \hat v, \hat \sa]=(0,0,0).$$
This implies 
\begin{equation*}
L_*(\mu_3)(\hat v,\hat \sa)=-\hat z(0,e_2),\quad P_3\,\hat\sa =0,\quad (\hat\sa | e_2)=0.
\end{equation*}
As $\hat z(0,e_2)\in {\sf N}(L_*(\mu_3))$, we conclude
 that $(\hat v,\hat \sa)\in N(L_*^2(\mu_3))=N(L_*(\mu_3))$, due to the fact that
 $0$ is a semi-simple eigenvalue of $L_*(\mu_3)$, see part (b).
 Hence, $(\hat v,\hat \sa)=(0, \alpha_2 e_3 +\alpha_3 e_3),$ as 
 $\hat \sa\in {\sf N}(\lambda_* - \II)={\rm span}\{e_2, e_3\}$.
 From $P_3 \hat\sa=0$ follows $\alpha_3=0$, and the condition $(\hat\sa | e_2)=0$ implies $\alpha_2=0$.
 Hence, $(\hat v,\hat\sa)=(0,0)$. 
 Lastly, the condition $\hat z(0,e_2)=L_*(\mu_3)(0,0)=(0,0)$ gives $\hat z=0$.
By the Implicit Function Theorem, 
there exist  smooth functions 
$$[\mu\mapsto (z(\mu),v(\mu),\sa(\mu))]: \BB(\mu_3,\eps_0)\to  \R\times {_0H}^2_{q,\sigma}(\Omega)\times\R^3$$
such that
\begin{equation}
\label{G-mu=0}
\quad G(\mu, (z(\mu),v(\mu),\sa(\mu)))=(0,0,0), \quad \mu\in \BB(\mu_3,\eps_0),
\end{equation}
provided $\eps_0$ is chosen small enough.

Equation~\eqref{G-mu=0} can be stated equivalently as
\begin{equation}
\label{EV-problem-mu}
\begin{aligned}
z(\mu)\big( v(\mu) + \PP((\sa(\mu) -\omega(\mu))\times x)\big) + 2\PP({\sa_*}\times v(\mu)) - \upnu\PP \Delta v(\mu) &=0 &&\text{in} && \Omega, \\
v(\mu) &=0 &&\text{on} && \partial\Omega, \\
z(\mu)\II(\mu)\sa(\mu) + \sa_*\times \II(\mu)\sa(\mu) + (\sa(\mu)-\omega(\mu))\times\II(\mu)\sa_* &=0 &&\text{on}&& \R^3.\\
\end{aligned}
\end{equation}
Here we remind that $\II(\mu)={\rm diag}\,[\mu_1,\mu,\mu_3]$ and
\begin{equation}
\label{condition-mu}
\sa_*=\alpha_* e_3,\quad\II(\mu) \sa_*=\mu_3\sa_* ,  \quad (z(\mu_3),v(\mu_3),\sa(\mu_3))=(0,0,e_2).
\end{equation}
Taking the derivative of the last line of \eqref{EV-problem-mu} with respect to $\mu$ and using \eqref{condition-mu}
 gives
\begin{equation*}
\begin{aligned}
&z^\prime(\mu_3)\II(\mu_3)\sa(\mu_3)+ \sa_*\times (\II^\prime(\mu_3)\sa(\mu_3)+\II(\mu_3)\sa^\prime(\mu_3)) 
+(\sa^\prime (\mu_3)-\omega^\prime (\mu_3))\times\II(\mu_3)\sa_*\\
 &=\mu_3 z^\prime(\mu_3)e_2 + \alpha_* e_3 \times e_2 +  \alpha_* e_3 \times (\II(\mu_3)-\mu_3)\sa^\prime(\mu_3) 
 +\mu_3 \alpha_*e_3\times \omega^\prime(\mu_3)=0.\\
\end{aligned}
\end{equation*}
Observing that $(\II(\mu_3)-\mu_3)\sa^\prime(\mu_3))= (\mu_1-\mu_3)\sa_1^\prime(\mu_3)e_1$ we obtain the relations
\begin{equation}
\label{relation-third-line}
\begin{aligned}
1+\mu_3 \omega_2^\prime (\mu_3)=0, \quad
\mu_3 z^\prime(\mu_3) +\alpha_* (\mu_1-\mu_3)\sa_1^\prime(\mu_3) +\mu_3\alpha_* \omega_1^\prime(\mu_3)=0.
\end{aligned}
\end{equation}
Taking the derivative of the first two lines in \eqref{EV-problem-mu} and  again using~\eqref{condition-mu} results in 
\begin{equation*}
\label{EV-problem-derivative}
\begin{aligned}
z^\prime(\mu_3) \PP (\sa(\mu_3) \times x) + 2\PP({\sa_*}\times v^\prime (\mu_3)) - \upnu\PP \Delta v^\prime(\mu_3) &=0 &&\text{in} && \Omega, \\
v^\prime (\mu_3) &=0 &&\text{on} && \partial\Omega.
\end{aligned}
\end{equation*}
Multiplying the first equation by $v^\prime(\mu_3)$ and integrating over $\Omega$ then gives
\begin{equation*}
\begin{aligned}
0= z^\prime(\mu_3)(\sa(\mu_3)|\omega^\prime(\mu_3)) +\upnu |\nabla v^\prime(\mu_3)|
  = z^\prime(\mu_3) \omega_2^\prime(\mu_3) +\upnu |\nabla v^\prime(\mu_3)|,
\end{aligned}
\end{equation*}
where we used $\sa(\mu_3)=e_2$.
Combining this with the first relation in \eqref{relation-third-line} gives
\begin{equation*}
z^\prime(\mu_3)=\mu_3\upnu  |\nabla v^\prime(\mu_3)|.
\end{equation*}
Hence $z^\prime(\mu_3)>0$, provided we know that $\nabla v^\prime(\mu_3)\neq 0$.
So let us assume for the moment that $\nabla v^\prime(\mu_3)=0.$
Then $v^\prime(\mu_3)=0$ and then also $\omega^\prime (\mu_3)=0$,
contradicting the first relation in \eqref{relation-third-line}.

\smallskip
\noindent
We have, thus, proved that $z^\prime(\mu_3)>0$.
This implies that the eigenvalue $z(\mu)$  crosses  the origin  with positive speed along the real axis at the instant when $\mu=\mu_3$, 
and then will stay positive for $\mu\in (\mu_3,\mu_3+\eps_0)$.
But more is true: $z(\mu)$ exists for any $\mu>\mu_3$ and will remain positive.
(Existence follows from a standard perturbation result for simple eigenvalues).
If $z(\mu)$ was zero for some $\mu >\mu_3-\eps_0$, then necessarily $\mu=\mu_3$ by what was proved in part (b).

Let then $\mu_2>\mu_3$ be fixed, and consider  the ordered triple of
eigenvalues 
$$(\lambda_1,\lambda_2,\lambda_3)=(\mu_1,\mu_3,\mu_2).$$ 
Then the above argument
shows that the equilibrium $(0,\sa_*)\in \{0\}\times {\sf N}(\lambda_*-\II)$ with $\lambda_*=\lambda_2$ is unstable,
and $-L_*$ has exactly one positive eigenvalue, namely $z(\mu_2)$.
Finally, we observe that this assertion still holds true in case $\lambda_1=\lambda_2<\lambda_3$.
\medskip\\
\noindent
{\bf (f)} We will now consider the case where $\II(\mu)={\rm diag}\,[\mu,\mu_2,\mu_3]$ 
with $\mu_3<\mu_2$, $\mu\in (\mu_3-\eps_0,\mu_3+\eps_0)$. 
As in the proof of step (e) we assume $\sa_*=\alpha_*e_3$, so that $\II(\mu) \sa_*=\mu_3\sa_*$.
By the results in step (e) we then know that $-L_*(\mu)$   has $0$ and $z(\mu_2)>0$ as simple eigenvalues,
with  ${\sf N}(L_*(\mu))=\{0\}\times {\sf N}(\lambda_*-\II(\mu))$, while all the remaining eigenvalues are in $[{\rm Re}\,z]<0$
as long as $\mu<\mu_3$.
(This follows from the fact that no eigenvalue can cross the imaginary axis as long as $\mu<\mu_3$).
We can now repeat the arguments of step (e) to show that another eigenvalue $z(\mu)$ will cross the origin along the
real axis and become positive. For this we just ought to replace $e_2$ by $e_1$ in the proof given in step (e).
As a result, we conclude that two eigenvalues are now in $[{\rm Re}\,z]>0$ for any $\mu>\mu_3$.
Fixing $\mu_1> \mu_3$ and setting
$$(\lambda_1,\lambda_2,\lambda_3)=(\mu_3,\min\{\mu_1,\mu_2\},\max\{\mu_1,\mu_2\})$$
yields the assertion.
The proof of Theorem~\ref{th:spectrum} is now complete.
\end{proof}

 \section{Nonlinear Stability}\label{sec:stability}
In this section we will provide a complete characterization of the nonlinear stability properties of the nontrivial equilibria for the given problem. The approach is that of the {\em generalized principle of linearized stability} in \cite[Chapter 5]{PrSi16}.

The following result, which is also valid in the more general context of abstract quasilinear parabolic equations, is of interest in itself.

%%%%
\begin{proposition} \label{pro: stability-instability}
Suppose $(p,q,\mu)$ and $(p_j,q_j,\mu_j)$  satisfy the assumptions \eqref{assumptions-pq}, \eqref{assumptions-mu},
and \eqref{mu-j},
and suppose $u_*$ is an equilibrium of \eqref{eq:evolution}.
Then the following properties hold.
\begin{enumerate}
\setlength\itemsep{1mm}
\item[{\bf (a)}] $u_*$ is stable in the topology of $B^{2\mu-2/p}_{qp}(\Omega)\times\R^3  \Longleftrightarrow$  
$u_*$  is stable in the topology of $B^{2-2/p}_{qp}(\Omega)\times\R^3$. 
%\vspace{1mm}
\item[{\bf (b)}] $u_*$ is stable in the topology of $B^{2\mu_1-2/p_1}_{q_1p_1}(\Omega)\times\R^3  \Longleftrightarrow$  $u_*$ 
is stable in the topology of $B^{2\mu_2-2/p_2}_{q_2p_2}(\Omega)\times\R^3$. 

\end{enumerate}
\end{proposition}
%%%%%%%%%%%%%%
\begin{proof}
By Theorem~\ref{th:local_strong}(a), there are positive constants  $\tau=\tau(u_*)$, $\eta=\eta(u_*)$ and $c_1=c_1(u_*)$, corresponding to the
initial value $u_*$, such that
\begin{equation}\label{AA}
\|u(\cdot, w_0)-u_*\|_{\EE_{1,\nu}(0,2\tau)}\le c_1 |w_0-u_*|_{X_{\gamma,\nu}},
\quad \nu\in\{\mu,1\},
\end{equation}
for any $w_0\in B_{X_{\gamma,\nu}}(u_*,\eta)$.
Moreover, there is a constant $c_2=c_2(\tau,\mu)$ such that
\begin{equation}\label{AB}
\| u - u_*\|_{BC([\tau,2\tau], X_{\gamma,1})} \le c_2 \| u -u_*\|_{\EE_{1,\mu}[0,2\tau]} 
\end{equation}
for any function $u\in \EE_{1,\mu}(0,2\tau)$, see for instance \cite[page 228]{PrSi16}.
For later use we denote the embedding constant of $X_{\gamma,1}\hookrightarrow X_{\gamma,\mu}$
by $c_\mu$. Consequently,
\begin{equation}\label{AC}
B_{X_{\gamma,1}}(u_*, \alpha)\subset B_{X_{\gamma,\mu}}(u_*, c_\mu\alpha).
\end{equation}
\medskip\noindent
{\bf (a)}
Suppose $u_*$ is stable in $X_{\gamma,\mu}$. 
Let $\eps>0$ be given and set $\eps_\mu:=\min\{\eps/(c_1c_2),\eta\}.$
By assumption, there is a number $\delta_\mu$ 
such that every solution of \eqref{eq:evolution} with initial value $u_0\in B_{X_{\gamma,\mu}}(u_*,\delta_\mu)$
exists globally and satisfies
\begin{equation}\label{AD}
 |u(t,u_0)-u_*|_{X_{\gamma,\mu}}<\eps_\mu, \quad\text{for all} \;\ t\ge 0.
 \end{equation}
Next, we chose $\delta \in (0,\delta_\mu/c_\mu]$ sufficiently small such that
\begin{equation}\label{AE}
|u(t,u_0)-u_*|_{X_{\gamma,1}}< \eps,\quad\text{for all} \;\; t\in [0,\tau],\;\; u_0\in B_{X_{\gamma,1}}(0,\delta).
\end{equation}
Here we note that \eqref{AE} follows from continuous dependence on the initial data, see \eqref{AA}.
As a consequence of \eqref{AC}, every solution $u(\cdot, u_0)$ with 
$u_0\in B_{X_{\gamma,1}}(u_*,\delta)$ exists globally
and satisfies \eqref{AD} as well as \eqref{AE}.
Next we will show by recursion that
 $u(t,u_0)\in B_{X_{\gamma,1}}(u_*,\eps)$ for all $t\ge 0$.
 Suppose we have already shown that
 $|u(t,u_0)-u_*|_{X_{\gamma,1}}< \epsilon$ for $t\in [0,(k+1)\tau]$ and $k\in\N_0.$
 We note that the case $k=0$ is exactly \eqref{AE}.
From the definition of $\eps_\mu$ and \eqref{AA}-\eqref{AB} as well as \eqref{AD} follows
 \begin{equation}
 |u(k\tau +s,u_0)-u_*|_{X_{\gamma,1}}\le c_1c_2 |u(k\tau,u_0) -u_*|_{X_{\gamma,\mu}}<\eps,\quad \tau\le s\le 2\tau. 
 \end{equation}
 Since this step works for any $k\in\N$, we obtain stability of $u_*$ in $X_{\gamma,1}$.
 \medskip\\
 \noindent
Suppose that $u_*$ is stable in $X_{\gamma,1}$. Let $\eps>0$ be given
 and set $\eps_1= \eps/c_\mu$.  By the stability assumption, there exists a number $\delta_1$ such that
every solution of \eqref{eq:evolution} with initial value $u_0\in B_{X_{\gamma,1}}(u_*,\delta_1)$
exists globally and satisfies
\begin{equation}\label{AF}
 |u(t,u_0)-u_*|_{X_{\gamma,1}}<\eps_1, \quad\text{for all} \;\ t\ge 0.
 \end{equation}
Next, by continuous dependence on initial data, we can chose $\delta\in (0,\eta)$ sufficiently small such that
\begin{equation}\label{AG}
|u(t,u_0)-u_*|_{X_{\gamma,\mu}}< \delta_1/(c_1c_2),\quad\text{for all} \;\; t\in [0,\tau],\;\; u_0\in B_{X_{\gamma,\mu}}(u_*,\delta).
\end{equation}
It follows from \eqref{AA}-\eqref{AB} and \eqref{AG} that
$|u(\tau,u_0) - u_*|_{X_{\gamma,1}} \le c_1c_2 |u_0-u_*|_{X_{\gamma,\mu}}<\delta_1$,
for all  $u_0\in B_{X_{\gamma,1}}(u_*,\delta)$.
Hence, by \eqref{AC} and \eqref{AF},
 $u(t,u_0)\in B_{X_{\gamma,\mu}}(u_*,\eps)$ for
any initial value $u_0\in B_{X_{\gamma,1}}(u_*,\delta)$.

\medskip\noindent
{\bf (b)} As the arguments are similar to part (a), we will only sketch the proof.
Suppose $u_*$ is stable in the topology of $B^{2\mu_1-2/p_1}_{q_1p_1}(\Omega)\times\R^3 $.
Let 
$$K_1: {_0B}^{2\mu_1-2/p_1}_{q_1p_1,\sigma}(\Omega)\times\R^3 
\hookrightarrow {_0B}^{2\mu_2-2/p_2}_{q_2p_2,\sigma}(\Omega)\times\R^3$$
be the embedding map, and $k_1$ its norm.
By continuous dependence on initial data, for any  $\delta^\prime>0$ there exists  $\delta^{\prime \prime}>0$  and $\tau>0$ 
such that $u(t,u_0)$ lies in $B_2(u_*,\delta^\prime)$
for all $t\in [0,\tau]$ and all initial values $u_0$ in $B_2(u_*,\delta^{\prime\prime})$,
where $B_2(u_*,\alpha)$ denotes the ball of radius $\alpha$ with respect to the topology of 
${B}^{2\mu_2-2/p_2}_{q_2p_2}(\Omega)\times\R^3$.
By  \eqref{AA}--\eqref{AB}, see also \eqref{EE-j}, the mapping
$$K_2: B_2(u_*,\delta^\prime)\to \EE_{1,\mu_2}(0,\tau)\to {_0B}^{2-2/p_2}_{q_2p_2,\sigma}(\Omega)\times\R^3,\quad 
u_0\mapsto u(\cdot,u_0)\mapsto u(\tau,u_0)
$$ 
is continuous, and the embedding $K_3: {_0B}^{2-2/p_2}_{q_2p_2,\sigma}(\Omega)\times\R^3
\hookrightarrow {_0B}^{2\mu_1-2/p_1}_{q_1p_1,\sigma}(\Omega)\times\R^3$
is so as well.
Let $\eps>0$ be given, choose $\eps_1:=\eps/k_1$ and let $\delta_1\in (0,\eps_1)$ be so that
the stability condition holds for all $w_0$ in $B_1(u_*,\delta_1)$, where 
$B_1(u_*,\eps_1)$ is the ball  with radius $\eps_1$ defined by the topology of 
$B^{2\mu_1-2/p_1}_{q_1p_1}(\Omega)\times\R^3 $.  
Using the properties of  $K_i$, one shows that there exists a number $\delta>0$ such that
$u(\tau,u_0)\in B_1(u_*,\delta_1)$ for all $u_0\in B_2(u_0,\delta)$.
Therefore, $u(s, u(\tau,u_0))$ exists for all $s\ge 0$ and stays in $B_1(0,\eps_1)$,
and stability in ${_0B}^{2\mu_2-2/p_2}_{q_2p_2,\sigma}(\Omega)\times \R^3$ follows.

\medskip\noindent
Lastly, suppose $u_*$ is stable in the topology of $B^{2\mu_2-2/p_2}_{q_2p_2}(\Omega)\times\R^3 $.
By part (a), we conclude that $u_*$ is also stable in the  
topology of $B^{2-2/p_2}_{q_2p_2}(\Omega)\times\R^3 $.
Stability in the topology of $B^{2\mu_1-2/p_1}_{q_1p_1}(\Omega)\times\R^3 $
can then be obtained from the continuity of the mapping 
\begin{equation*}
\begin{aligned}
 &B_1(u_*,\delta_1)  \to  B_2(u_*,k_1\delta_1)  \to  \EE_{1,\mu_2}(0,\tau)  \to  {_0B}^{2-2/p_2}_{q_2p_2,\sigma}(\Omega)\times\R^3,\\
& u_0  \mapsto  u(\cdot,u_0)  \mapsto  u  (\tau,u_0),
 \end{aligned}
 \end{equation*}
by similar arguments as above.
\end{proof}

%%%%%%%%%%%%
\begin{theorem}\label{th:stability-instability}
Suppose $p,q$ and $\mu$ satisfy the assumptions of Theorem~\ref{th:local_strong}.
Let $(0,\sa_*)\in \cE$ with $\sa_*\neq 0$ be an equilibrium of~\eqref{problem},
and recall that $\sa_*\in {\sf N}(\lambda_*-\II)$. 
Then the following statements hold. 
\begin{itemize}
\setlength\itemsep{1mm}
\item[(i)] If $\lambda_*=\max\{\lambda_1,\lambda_2,\lambda_3\}$, then $(0,\sa_*)$ is stable
in ${_0B}^{2\mu-2/p}_{qp,\sigma}(\Omega)\times \R^3$. 
Moreover, there exists $\delta>0$ such that the unique solution $(v(t),\sa(t))$ of \eqref{eq:evolution} with initial value $(v_0,\sa_0) \in {_0B}^{2\mu-2/p}_{qp,\sigma}(\Omega)\times \R^3$ satisfying 
\[
|(v_0, \sa_0-\sa_*)|_{B^{2\mu-2/p}_{qp}(\Omega)\times\R^3}<\delta
\]
exists on $\R_+$ and converges exponentially fast 
to some $(0,\bar\sa)\in \cE$
in the topology of  
$$H_q^{2\alpha}(\Omega)\times \R^3,\quad\text{for any $\alpha\in [0,1)$}.$$
\item[(ii)] If $\lambda_*\ne\max\{\lambda_1,\lambda_2,\lambda_3\}$, then $(0,\sa_*)$ is unstable in 
${_0B}^{2\mu-2/p}_{qp,\sigma}(\Omega)\times \R^3.$ 
\end{itemize}
\end{theorem}
%%%%%%
\begin{proof}
Recall that the nonlinearity $F$ of equation~\eqref{eq:evolution} has the bilinear structure $F(u)=G(u,u)$ given in \eqref{eq:bilinear_nonlinearity}, with $G:X_\beta\times X_\beta\to X_0$  bounded for  $\beta=\frac{1}{4}\big(1+\frac{3}{q}\big).$  
Let $\alpha\in (0,1)$ be given and choose $p_2$ large enough so that
$$B^{2-2/p_2}_{qp_2}(\Omega)\hookrightarrow H^{2\beta}_q(\Omega)\cap H^{2\alpha}_q(\Omega),
\quad 1/p_2<1-\beta.$$
Hence we conclude that 
$X_{\gamma,1}=(X_0,X_1)_{1-1/p,p}\hookrightarrow X_\beta,$
and $F\in C^{1-}(X_{\gamma,1},X_0)$.
Moreover, using Proposition \ref{pro:L*-calculus}, we infer that $L$ has the property of maximal $L_p$-regularity. 
With $\mu_2=\mu-(1/p-1/p_2)$ 
the stability result in (i) follows from Theorem~\ref{th:spectrum}, Proposition~\ref{pro: stability-instability},
and  \cite[Theorem 5.3.1]{PrSi16},  since $(0,\sa_*)$ is  normally stable.
In addition, by \eqref{AA}--\eqref{AB}, we conclude that for every $\delta^\prime$ 
there are positive numbers $\delta,\tau$, such that
$| u(\tau,u_0)-u_*|_{X_{\gamma,1}}<\delta^\prime$ for all  $ |u_0-u_*|_{X_{\gamma,\mu}}<\delta.$
The remaining assertion in (i) then follows from \cite[Theorem 5.3.1]{PrSi16}.

By  Theorem \ref{th:spectrum} and \cite[Theorem 5.4.1]{PrSi16} we know that
$(0,\sa_*)$ is unstable in ${_0B}^{2-2/p_2}_{qp_2,\sigma}(\Omega)\times \R^3$.
But then $(0,\sa_*)$ is also unstable in ${_0B}^{2\mu-2/p}_{qp,\sigma}(\Omega)\times \R^3$,
as we would otherwise obtain a contradiction by  Proposition~\ref{pro: stability-instability}.
\end{proof}
\begin{remarks}
{\bf (a)} Choosing $\mu=\mu_c$ in Theorem~\ref{th:stability-instability}, 
we obtain the assertions of the theorem for initial values
$(v_0,\sa_0)\in {_0B}^{3/q-1}_{qp,\sigma}(\Omega)\times\R^3.$

\smallskip\noindent
{\bf (b)} Choosing $q=q=2$ in Theorem~\ref{th:stability-instability}, 
we obtain the assertions of the theorem for initial values
 ${_0H}^{2\alpha}_{2,\sigma}(\Omega)\times\R^3$ for $\alpha\in [1/4, 1/2]$, where
$${_0H}^{1/2}_{2,\sigma}(\Omega):=(L_{2,\sigma}(\Omega), {_0H}^2_{2,\sigma}(\Omega))_{1/4,2}.$$
\end{remarks}

\goodbreak

\section{Long-time behavior}\label{se:long-time}
We conclude our paper by analyzing the long-time behavior of Leray-Hopf weak solutions. 
Let us first recall the definition of weak solution. 

\begin{defi}%\label{def:weak}
The couple $(v,\sa)$ is a {\em weak solution} \`a la Leray-Hopf of \eqref{problem} if the following conditions are satisfied. 
\begin{enumerate}
\setlength\itemsep{1mm}
\item $v\in C_w ([0,\infty);L_{2,\sigma}(\Omega))\cap L_\infty((0,\infty);L_{2,\sigma}(\Omega))
\cap L_2((0,\infty);{_0H}^1_2(\Omega))$. 
%\vspace{2mm}
\item $\sa\in C^0([0,\infty))\cap C^1((0,\infty))$.  
%\vspace{2mm}
\item $(v,\sa)$ satisfies the equations of motions \eqref{problem} in the distributional sense and the boundary condition in the trace sense;
%\vspace{2mm}
\item The {\em strong energy inequality} holds: 
\[
{\sf E}(v(t),\sa(t))+\upnu \int^t_s|\nabla v(\tau)|_\Omega^2\; d\tau\le{\sf E}(v(s),\sa(s)),
\]
for all $t\ge s$ and a.a. $s\ge 0$ including $s=0$. 
\end{enumerate}
\end{defi}

The class of the above solutions is nonempty for initial data having finite kinetic energy. 
Moreover, such weak solutions become strong after a sufficiently large time and the fluid relative velocity tends to zero (in a suitable topology) as time approaches infinity. We summarize all these results in the following theorem.

%%%%%%%%%%%%%%%
\begin{theorem}\label{th:weak-strong}
For any initial value $(v_0, {\sf a}_0)\in L_{2,\sigma}(\Omega)\times \R^3$, 
there exists at least one weak solution $(v,\sa)$ \`a la Leray-Hopf. 
Moreover, there exists a time $\tau>0$  such that  
\begin{equation*}%\label{eq:regularity_c}
\begin{split}
v\in H^1_{p,\,loc}([\tau,\infty);L_{q,\sigma}(\Omega))
\cap L_{p,\,loc}([\tau,\infty);{_0H}^2_{q,\sigma}(\Omega)),\quad \sa\in C^1([\tau,\infty);\R^3),
\end{split}\end{equation*} 
for each $p\in [2,\infty)$, $q\in (1,3)$.
 Finally, $v\in C([\tau,\infty);B^{2-2/p}_{qp}(\Omega))$, and 
\begin{equation}\label{eq:bvto0}
\lim_{t\to \infty}|v(t)|_{H^{2\alpha}_q}=0,\quad \text{for any }\alpha\in[0,1).
\end{equation}
\end{theorem}
%%%%%%%%%%
\begin{proof}
The existence of a weak solution $(v,\sa)$ corresponding to any initial condition $(v_0, {\sf a}_0)\in L_{2,\sigma}(\Omega)\times \R^3$ has been shown in \cite[Theorem 3.2.1]{Ma12}. 
 
By \cite[Proposition 1]{DGMZ16}, there exists $t_0>0$ such that 
\[\begin{split}
&v\in C([t_0,\infty);{_0H}^1_{2,\sigma}(\Omega))\cap L_{2,loc}([t_0,\infty);H^2_{2,\sigma}(\Omega))\cap H^1_{2,loc}([t_0,\infty);L_{2,\sigma}(\Omega)),
\\
&\sa\in C^1([t_0,\infty);\R^3). 
\end{split}\]
Moreover, there exists a pressure field $p\in L_{2,loc}([t_0,\infty);H^1_2(\Omega))$ such that the triple $(v,\sa,p)$ satisfies 
\eqref{problem} a.e. in $\Omega\times (t_0,\infty)$. In addition, 
\begin{equation}\label{eq:gradvto0}
\lim_{t\to \infty}|v(t)|_{H^1_2}=0.
\end{equation}
For $p\ge 2$ and $q\in[2,3)$, let
$\mu ={1}/{p}+ {3}/{2q}- {1}/{4},$
yielding
$H^1_2(\Omega)\hookrightarrow B^{2\mu-2/p}_{qp}(\Omega)$.
Next, we take $(v(t_0), {\sa}(t_0))\in {_0H}^1_{2,\sigma}(\Omega)\times \R^3$ as initial condition for 
a strong solution $(\tilde v,\tilde \sa)$ to \eqref{problem} on $[t_0, t_1]$, $t_1 \in (t_0, t_+)$, in
the class 
\[\begin{split}
&H^1_{p,\mu}((t_0,t_1);L_{q,\sigma}(\Omega)\times \R^3)\cap L_{p,\mu}((t_0,t_1);{_0H}^2_{q,\sigma}(\Omega)\times \R^3)
\\
&\cap H^1_{2}((t_0,t_1);L_{2,\sigma}(\Omega)\times \R^3)\cap L_{2}((t_0,t_1);{_0H}^2_{2,\sigma}(\Omega)\times \R^3), 
\end{split}\]
as from Theorem~\ref{th:local_strong}(b) (see also Remark~\ref{remark-2}(b)). 
Here we have used the notation
$$u\in L_{p,\mu}((t_0,T);X) \Leftrightarrow (t-t_0)^{1-\mu}u\in L_p((t_0,T);X)$$
and $u\in H^1_{p,\mu}((t_0,T);X) \Leftrightarrow u,\dot u\in L_{p,\mu}((t_0,T);X)$
for $0\le t_0<T\le \infty$ and $X$ a Banach space.

Since uniqueness of solutions holds in the above class, necessarily $v\equiv \tilde v$ and $\sa\equiv \tilde \sa$ on $[t_0, t_1]$, 
$t_1 \in (t_1, t_+)$.  Hence we conclude from  \eqref{momentum-conserved}, \eqref{eq:gradvto0}, 
and Theorem~\ref{th:local_strong}(d) that $t_+=\infty$.
Choosing $\tau>t_0$, the first assertion follows from the fact that $L_{p,\mu}((t_0, T);X)|_{[\tau,T]}\subset L_p((\tau,T);X)$
for any $\tau \in (t_0,T)$. Here we also note that the range $q\in (1,2)$ is admissible, since $\Omega$ is bounded.

Let us conclude our proof by showing \eqref{eq:bvto0}.
Let $\alpha\in [0,1)$ be given. By choosing $p$ sufficiently large we have
$B^{2-2/p}_{qp}(\Omega)\hookrightarrow  H^{2\alpha}_q(\Omega).$ 
Let $p$ be fixed so that this embedding holds.
Moreover, the embedding $H^1_2(\Omega)\hookrightarrow B^{2\mu-2/p}_{qp}(\Omega)$ holds with $\mu =1/p+ 3/2q-1/4.$ 
Hence, \eqref{eq:gradvto0} implies 
\begin{equation}
\label{to0-inB}
\lim_{t\to \infty}\norm{v(t)}_{B^{2\mu-2/p}_{qp}}=0.
\end{equation}
We can now conclude from  \eqref{momentum-conserved} 
that $(v,{\sf a})([\tau,\infty))$ is relatively compact in 
\begin{equation*}
{_0B}^{2\bar \mu-2/p}_{qp,\sigma}(\Omega)\times \R^3, \quad\text{for any } \bar\mu\in [\mu_{\rm crit},\mu), 
\end{equation*}
where $ \mu_{\rm crit}=1/p+3/2q-1/2.$
Theorem 5.7.1 in \cite{PrSi16} shows that $(v,{\sf a})([\tau+1,\infty))$ is 
compact, and hence also bounded, in $B^{2-2/p}_{qp}(\Omega)\times \R^3.$ 
Choosing $\theta$ such that
\[
( B^{2\mu-2/p}_{qp}(\Omega),B^{2-2/p}_{qp}(\Omega))_{\theta,2}= B^{2\alpha}_{q2}(\Omega)\hookrightarrow H^{2\alpha}_q(\Omega)
\]
yields \eqref{eq:bvto0} by interpolation.
\end{proof} 
%%%%%%%%%%%%%%%
In the following, we state our main result about the long-time behavior of solutions to \eqref{problem}. 
\begin{theorem}\label{th:convergence}
Let $(v,\sa)$ be a weak solution corresponding to  an initial condition $(v_0,\sa_0)\in L_{2,\sigma}(\Omega)\times \R^3$ with $\sa_0\ne 0$. 
Then there exists $\bar \sa\in \R^3$ with $|\II \sa_0|= |\II \bar\sa|$, such that 
\begin{equation*}
(v(t),\sa(t))\to (0, \bar\sa)\quad\text{in} \;\; H^{2\alpha}_q(\Omega)\times \R^3, \;\; \text{for any }\alpha\in[0,1),\;  q\in (1,3), 
\end{equation*}
at an exponential rate.
%%%%%%%%%%
\end{theorem}
\begin{proof}

We recall that, by Proposition \ref{prop:equilibria}(b),  ${\sf E}$  is a strict Lyapunov function for \eqref{eq:evolution}. 
Moreover, the nonzero equilibria have been characterized in Theorem \ref{th:spectrum} to be either normally stable or normally hyperbolic. Now, fix $\alpha\in (0,1)$. As in the proof of Theorem \ref{th:stability-instability}, we choose $p$ large enough so that
$$B^{2-2/p}_{qp}(\Omega)\hookrightarrow H^{2\beta}_q(\Omega)\cap H^{2\alpha}_q(\Omega),
\quad 1/p<1-\beta.$$ Hence $X_{\gamma,1}=(X_0,X_1)_{1-1/p,p}\hookrightarrow X_\beta,$
and $F\in C^{1-}(X_{\gamma,1},X_0)$. 
We also recall that $L$ has the property of maximal $L_p$-regularity. 
By Theorem \ref{th:weak-strong}, we know that $v(t)\in {_0B}^{2-2/p}_{pq,\sigma}(\Omega)$ for all $t\ge\tau$, 
and exponential convergence to an equilibrium, in the stated topology, follows from \cite[Theorem 5.7.2]{PrSi16} and the above embedding. 
\end{proof}

\begin{remark}{\rm
Suppose $\sa_0=0$. By Theorem \ref{th:weak-strong}, any weak solution $(v(t),\sa(t))$ corresponding to $(v_0,0)$ exists globally in time.  
Moreover, by conservation of total angular momentum \eqref{momentum-conserved}, necessarily $\sa(t)= 0$ for all times. 
Then, again by Theorem \ref{th:weak-strong}, $v(t)$ is a strong solution of 
\[
\dot v+A_1v=f(v),\quad v(0)=v(\tau),
\]
where  $A_1$ has been defined in \eqref{eq:A1} and $f(v)=(I+C)\PP(-v\cdot \nabla v+2\;\omega\times v)$. By Proposition \ref{pro:L-calculus} and \cite[Corollary 2.2 (iii)]{PrSiWi18} it follows that the rate of convergence in \eqref{eq:bvto0} is in fact exponential. A similar result has been obtained, with a completely different proof, in \cite[Theorem 5.6]{SiTa}. }
\end{remark}

\bigskip\noindent
{\bf Acknowledgment:}
The second and third author would like to thank Martin and Burga Simonett for their hospitality
while visiting Lohn, where part of this manuscript was written.

\end{document}